\documentclass{amsart}
\makeatletter
\@namedef{subjclassname@2020}{%
  \textup{2020} Mathematics Subject Classification}
\makeatother 

\usepackage{amsthm,amssymb,amsfonts,latexsym,mathtools,thmtools, mathrsfs}
\usepackage[T1]{fontenc}
\usepackage{mathrsfs}
\usepackage{tikz-cd} 
\usepackage{enumitem} 
\usepackage{hyperref} 
\usepackage[ruled,vlined]{algorithm2e}
\usepackage[all]{xy}
\hypersetup{
    colorlinks=true,
    linkcolor=blue,
    filecolor=blue,      
    urlcolor=blue,
    linktocpage=true
}

\newtheorem{theorem}{Theorem}[section]
\newtheorem{lemma}[theorem]{Lemma}

\theoremstyle{definition}
\newtheorem{definition}[theorem]{Definition}
\newtheorem{example}[theorem]{Example}

\theoremstyle{remark}
\newtheorem{remark}[theorem]{Remark}

\newtheorem{proposition}[theorem]{Proposition}

\newtheorem{reduction}{Reduction}[section]
\newtheorem{idea}{Idea}[section]
\newtheorem{problem}{Problem}[section]

\numberwithin{equation}{section}

\begin{document}

\title[On SAGBI bases theory of skew PBW extensions]{On SAGBI bases theory of skew Poincar\'e-Birkhoff-Witt extensions}

\author{Y\'esica Su\'arez}
\address{Universidad Nacional de Colombia - Sede Bogot\'a. Universidad Pedag\'ogica y Tecnol\'ogica de Colombia - Sede Tunja}
\curraddr{Campus Universitario}
\email{ypsuarezg@unal.edu.co}
\thanks{}

\author{Armando Reyes}
\address{Universidad Nacional de Colombia - Sede Bogot\'a}
\curraddr{Campus Universitario}
\email{mareyesv@unal.edu.co}

\thanks{The authors were supported by the research fund of Faculty of Science, Code HERMES 53880, Universidad Nacional de Colombia - Sede Bogot\'a, Colombia.}



\date{}

\dedicatory{Dedicated to our dear Professor Oswaldo Lezama}

\begin{abstract}

In this paper we present a first approach toward a \texttt{SAGBI} bases theory of skew Poincar\'e-Birkhoff-Witt extensions, and investigate the problem of polynomial composition for \texttt{SAGBI} bases of subalgebras of these extensions.

\end{abstract}

\maketitle

\section{Introduction}

Gordan \cite{Gordan1900} introduced the idea of {\em Gr\"obner basis} in 1900 and sixty five years later these bases for commutative polynomial rings over fields were defined and developed by Buchberger in his Ph.D. Thesis \cite{Buchberger1965} under the direction of Gr\"obner. In this commutative setting, Shannon and Sweedler \cite{ShannonSweedler1989} solved the membership problem in subalgebras transforming it into a problem of ideals.  Almost at the same time, Kapur and Madlener \cite{KapurMadlener1989} presented a procedure to compute a canonical basis for finitely presented subalgebras. In that paper they moved away a bit from what was worked by Shannon and Sweedler \cite{ShannonSweedler1989}, generalizing the theory of the Buchberger's algorithm using the {\em term rewriting method}. Nevertheless, this direct method has a disadvantage because for some orders on indeterminates and terms, the completion algorithm may not terminate and thus generate an infinite canonical basis, something that does not happen in the approach of Shannon and Sweedler. One year later, Robbiano and Sweedler \cite{RobbianoSweedler1990} studied the analogue of Gr\"obner bases for subalgebras of commutative polynomial rings, and used for the first time the term \texttt{SAGBI} ({\em Subalgebras Analogue to Gr\"obner Basis for Ideals}). They showed explicitly that the \texttt{SAGBI} theory is not simply a formal translation of Buchberger's theory from ideals to subalgebras. The theory of Gr\"obner bases of ideals of a subalgebra in a polynomial ring was developed by Miller \cite{Miller1998}, while Lezama and Mar\'in \cite{LezamaMarin2009} used \texttt{SAGBI} bases to determine the equality of subalgebras based on the the ideas presented by Kreuzer and Robbiano \cite{KreuzerRobbiano2005}. 

In the noncommutative case, Nordbeck in his Ph.D. Thesis \cite{NordbeckPhD2001} (see also \cite{Nordbeck1998, Nordbeck2002})  studied \texttt{SAGBI} bases for subalgebras of the free associative algebra $\Bbbk\{ X\}$ in the set of indeterminates $X$ over a field $\Bbbk$. He defined the process of subalgebra reduction, formulated generalizations of the standard Gr\"obner bases techniques for the test whether bases are canonical and for the completion procedure of constructing canonical bases, and discussed the special case of homogeneous subalgebras. More recently, for the $G$-{\em algebras} introduced by Apel \cite{Apel1988} and studied by Levandovskyy in his Ph.D. thesis \cite{Levandovskyy2005}, Khan et al. \cite{Khanetal2019} established for the first time the theory of \texttt{SAGBI} bases for these algebras and developed a computational criterion for its construction.

Having in mind all above treatments, in this paper we present a first approach toward a \texttt{SAGBI} bases theory of the family of noncommutative rings known as {\em skew Poincar\'e-Birkhoff-Witt} ({\em SPBW} for short) {\em extensions} which were introduced by Gallego and Lezama \cite{GallegoLezama2010}. Our motivation is due to that Gallego in her Ph.D. Thesis \cite{Gallego2015PhD} (and related papers \cite{Gallego2016FG, GallegoLezama2010}) developed the Gr\"obner basis theory for skew PBW extensions, and studied different homological properties of projective modules over these extensions \cite{Gallego2016, GallegoLezama2017}, so that a natural task is to investigate the problem of \texttt{SAGBI} bases for these objects. This is our aim in this paper: present a first approach toward a theory of \texttt{SAGBI} bases for skew PBW extensions over $\Bbbk$-algebras. We will see that many of the proofs we give are straightforward modifications of the original proofs presented by Nordbeck, but we include some of them for the sake of completeness and because of the occasional differences.

The paper is organized as follows. In Section \ref{Freealgebras} we recall some preliminaries and results on \texttt{SAGBI} bases following Nordbeck's ideas \cite{Nordbeck1998, NordbeckPhD2001}. Section \ref{SAGBIbasesSPBW} contains the results of the paper. We start by considering some preliminaries on SPBW extensions, formulate the notion of {\em reduction} which is necessary in the characterization of \texttt{SAGBI} bases, and then establish an algorithm to find the \texttt{SAGBI} normal form of an element. Then, we define what a \texttt{SAGBI} basis is, and formulate a criterion to determine when a subset of a SPBW extension over a field is a \texttt{SAGBI} basis. In addition, we establish an algorithm to find a \texttt{SAGBI} basis from a subset contained in a subalgebra of a SPBW extension. We illustrate our results with examples concerning algebras appearing in noncommutative algebraic geometry. Finally, in Section \ref{SPBWBasesundercomposition} we investigate the problem of polynomial composition for \texttt{SAGBI} bases of subalgebras of SPBW extensions.

\section{Free algebras}\label{Freealgebras}

Nordbeck in his Ph.D. Thesis \cite{NordbeckPhD2001} (see also \cite{Nordbeck1998, Nordbeck2002}) studied \texttt{SAGBI} bases for subalgebras of the free associative algebra over a field $\Bbbk$. In this section, we recall his key ideas.

\subsection{Basic definitions}

Let $X = \left\{x_1, \ldots, x_n\right\}$ be a finite alphabet, and \linebreak $\Bbbk\left\{x_1, \ldots, x_n\right\}$ the {\em free associative algebra over} $\Bbbk$. The symbol $W$ means the set of all {\em words} in $X$ including the empty word 1 (i.e. $W=X^*$, the free monoid generated by $X$). For $H$ a subset of $\Bbbk \left\{x_1, \ldots, x_n\right\}$, the subalgebra $S$ of $\Bbbk \left\{x_1, \ldots, x_n\right\}$ generated by $H$ is denoted by $\Bbbk \{H\}$. In other words, {\em the elements of} $S$ {\em are the polynomials in the set of formal and noncommutative indeterminates} $H$ {\em viewed as elements of} $\Bbbk\left\{x_1, \ldots, x_n\right\}$. It is assumed that the coefficient of such monomials is the identity $1$. With this description, the monomials are not general words appearing in $\Bbbk \left\{x_1,  \ldots, x_n\right\}$. As it is clear, the elements of $\Bbbk \subset S$ correspond to the constant polynomials.

\begin{example}[{\cite[Example 1]{Nordbeck1998}}]\label{Nordbeck1998Example1}
    Consider the subset 
    $$
    H := \{ h_1 = yx^2 + 1,\ h_2 = 2xy - y,\ h_3 = yx\}
    $$ 
    
    of the free algebra $\Bbbk \{x, y\}$. An example of a monomial $m(H) \in \Bbbk \{X\}$ is given by 
    $$
    h_1 h_2 = (yx^2 + 1) (2xy - y) = 2yx^3y - yx^2y + 2xy - y.
    $$
\end{example}

For Nordbeck \cite{Nordbeck1998}, an {\em admissible order} $\prec$ on a set $W$ is a well-order preserving multiplication, that is, $ f \prec g $ implies $h f k \prec h g k$ for all $f, g, h, k \in W$, such that the smallest word is the unity 1. It follows that every infinite sequence $u_1 \succ u_2 \succ \dotsb \succ $ in $W$ stabilizes. He used the order \texttt{deglex} (see Example \ref{Monomialordersexamples}(ii) for its definition), so if terms with identical words are collected together using the operations over $\Bbbk$, then with every non-zero element $f\in \Bbbk \{x_1, \dotsc, x_n\}$ we can associate its {\em leading word} ${\rm lw}(f)$, that is, the word in $f$ that is larger (relative the order $\prec$) than every other word occurring in $f$. The {\em leading term} ${\rm lt}(f)$ of $f$ is the leading word times its coefficient. For instance, in Example \ref{Nordbeck1998Example1}, using \texttt{deglex} we get that
$$
{\rm lw}(m(H)) = yx^3y \quad {\rm and} \quad {\rm lt}(m(H)) = 2yx^3y.
$$ 

As expected, for $H \subseteq \Bbbk \{X\}$ we have that ${\rm lw}(H) := \left\{{\rm lw}(h) \mid h\in H\right\}$. 

\begin{definition}[{\cite[Definition 1]{Nordbeck1998}}]\label{Nordbeck1998Definition1}
Let $H \subset \Bbbk\{x_1, \dotsc, x_n\}$, and let $\sum\limits_{i=1}^t k_i m_i(H)$, $k_i \in \Bbbk,\ m_i \in \Bbbk\{ H\}$, be a sum of monomials. The {\em height} of the sum is given by 
$$
\max \left\{\operatorname{lw}(m_i(H)) \mid 1 \leq i \leq t\right\},
$$ 

where the maximum is taken relative to the order in $\Bbbk\left\{x_1, \ldots, x_n\right\}$. The {\em breadth} of the sum is the number of $i$'s such that $\operatorname{lw}(m_i(H))$ is equal to the height.
\end{definition}

As one can see, the leading word of $f = \sum\limits_{i=1}^t k_i m_i(H)$ can be smaller than the height of $\sum\limits_{i=1}^t k_i m_i(H)$. As a matter fact, this happens if and only if some words larger than ${\rm lw}(f)$ cancel in the sum, and the breadth of the sum is then necessarily at least two.

We present the definition of \texttt{SAGBI} basis in the case of free associative algebras (c.f. \cite[Definition 6.6.2]{KreuzerRobbiano2005}).

\begin{definition}[{\cite[Definition 2]{Nordbeck1998}}]\label{Nordbeck1998Definition2}
Let $S$ be a subalgebra of $\Bbbk\{x_1, \dotsc, x_n\}$. A subset $H\subset S$ is called a \texttt{SAGBI} basis for $S$ if for every $f\in S$, $f\neq 0$, there exists a monomial $m\in \Bbbk \{H\}$ such that ${\rm lw}(f) = {\rm lw}(m(H))$.
\end{definition}

By recalling that orders are preserved after multiplication, if 
$$
m(H) = h_{i_1} h_{i_2} \dotsb h_{i_t}
$$ 

with $h_{ij} \in H$, then 
$$
{\rm lw}(m(H)) = {\rm lw}(h_{i_1}) {\rm lw}(h_{i_2}) \dotsb {\rm lw}(h_{i_t}). 
$$

This means that an equivalent formulation of Definition \ref{Nordbeck1998Definition2} is that $H$ is a \texttt{SAGBI} basis if the leading word of every non-zero element in $S$ can be written as a product of leading words of elements in $H$. Notice that every subalgebra is a \texttt{SAGBI} basis for itself, so every subalgebra has a \texttt{SAGBI} basis. Also, if $H$ consists only of words (or terms), then $H$ is a \texttt{SAGBI} basis for the subalgebra $S$ generated by $H$ \cite[Example 2]{Nordbeck1998}.

An important fact is that the \texttt{SAGBI} property depends on which order we consider, as the following example illustrates (c.f. \cite[Example 6.6.7]{KreuzerRobbiano2005}).

\begin{example}[{\cite[Example 3]{Nordbeck1998}}]\label{Nordbeck1998Example3}
Let 
$$
H := \{h_1 = yz,\ h_2 = zy,\ h_3 = x-y\}\subseteq \Bbbk \{x, y, z\},
$$

and consider as $S$ the subalgebra generated by $H$ with the order \texttt{deglex} and $x \succ y \succ z$. Then ${\rm lw}(H) = \{yz, zy, x\}$, and it is easy to check (by using Proposition \ref{Nordbeck1998Proposition4} below) that $H$ is a \texttt{SAGBI} basis.

On the other hand, if we let $x\prec y \prec z$, then ${\rm lw}(H) = \{yz, zy, y\}$. However, 
$$
p(H) = h_1h_3 - h_3h_2 = yzx - xzy \in S
$$ 

and ${\rm lw}(p(H)) = yzx$ cannot be written as a product of words in ${\rm lw}(H)$. This means that $H$ is not a \texttt{SAGBI} basis for $S$.
\end{example}

Next, we present the process of reduction in \texttt{SAGBI} bases (c.f. \cite[Definition 6.6.16]{KreuzerRobbiano2005}).

\subsection{Reduction}

\begin{reduction}[{\cite[p. 141]{Nordbeck1998}}]\label{Nordbeck1998reduction}
The {\em reduction} of $f\in \Bbbk\{x_1, \dotsc, x_n\}$ over a subset $H\subset \Bbbk \{x_1, \dotsc, x_n\}$ is performed as follows:
\begin{enumerate}
    \item [\rm (1)] $f_0 = f$.
    \item [\rm (2)] If $f_i = 0$, or if there is no monomial $m \in \Bbbk \{H\}$ with ${\rm lw}(f_i) = {\rm lw}(m(H))$, then terminate. In case of termination this $f_i$ will be referred as the result of the reduction.
    \item [\rm (3)] Find a monomial $m_i \in \Bbbk \{H\}$ and $k_i \in \Bbbk$ such that ${\rm lt}(f_i) = {\rm lt}(k_i m_i(H))$. (This is possible since we have not terminated in step (2)). Now let $s_{i + 1} = s_i - k_im_i(H)$.
    \item [\rm (4)] Go to step (2) ($i + 1 \mapsto i$).
\end{enumerate}
\end{reduction}

When step (3) has been performed, the leading word of $s_{i+1}$ is strictly smaller than the leading word of $s_i$ (by the choice of $m_i$ and $k_i$). Now, since the order is well-founded, the reduction always terminates after a finite number of steps.

Notice that if $H$ is a finite set, then it is a constructive matter to determine whether a given word is a product of elements in ${\rm lw}(H)$, and hence the reduction is algorithmic. This is also the case if $H$ is infinite but sorted by e.g. the length of the leading words.

If the result (i.e. the last $f_i$) of a reduction of $f$ over $H$ is denoted $\overline{f^{H}}$, and if the reduction terminated after $t$ iterations of step (3) above, then 
\begin{equation}\label{Nordbeck1998(1)}
    f = \sum_{i = 0}^{t-1} k_im_i(H) + \overline{f^{H}}.
\end{equation}
If $t = 0$ the right hand side of (\ref{Nordbeck1998(1)}) is of course just $\overline{f^H}$. If step (3) is performed at least once we have, then
\begin{equation}\label{Nordbeck1998(2)}
    {\rm lw}(f) = {\rm lw}(m_0(H)) \succ {\rm lw}(m_1(H)) \succ \dotsb \succ {\rm lw}(m_{t-1}(H)),
\end{equation}
so the sum in the right hand side of (\ref{Nordbeck1998(1)}) is clearly of breadth one and height equal to ${\rm lw}(f)$. Nordbeck used these facts several times (in particular when $\overline{f^H} = 0$).

Notice that there are several different possibilities to choose the $m_i$'s in step (3), so the result of the reduction depends (in general) on how we choose these monomials.

It is interesting to consider the case when $\overline{s^H} = 0$ above. Nordbeck \cite[p. 142]{Nordbeck1998} said that $f$ {\em reduces to zero weakly over} $H$ if there exists one reduction (i. e. one choice of the $m_i$'s) with $\overline{s^H} = 0$, and that $s$ {\em reduces to zero strongly over} $H$ if every reduction (every choice) yields $\overline{s^H} = 0$. As he asserted, in most cases it does not matter which formulation we use, and we will then simply say that $s$ {\em reduces to zero over} $H$. By definition, $f = 0$ reduces to zero.

The following result is an application of \texttt{SAGBI} bases for the {\em Subalgebra Membership Problem}, that is, the decision whether an element $f\in \Bbbk \{x_1, \dotsc, x_n\}$ is in a given subalgebra (c.f. \cite[Theorem 6.6.25 (C1)]{KreuzerRobbiano2005}.

\begin{proposition}[{\cite[Proposition 1]{Nordbeck1998}}]\label{Nordbeck1998Proposition1}
    Let $H \subset \Bbbk \{x_1, \dotsc, x_n\}$ be a {\rm \texttt{SAGBI}} basis for the subalgebra $S$, and let $f \in \Bbbk \{x_1, \dotsc, x_n\}$. Then $f\in S$ if and only if $f$ reduces to zero over $H$.
\end{proposition}

As expected, if $H$ is a \texttt{SAGBI} basis for the subalgebra $S$, then $H$ generates $S$ \cite[Corollary 1]{Nordbeck1998}. We will simply say that $H\subset \Bbbk \{x_1, \dotsc, x_n\}$ is a \texttt{SAGBI} basis meaning that $H$ is a \texttt{SAGBI} basis for the subalgebra of $\Bbbk \{x_1, \dotsc, x_n\}$ generated by $H$.

\begin{remark}
    Notice that an arbitrary element of $\Bbbk \{x_1, \dotsc, x_n\}$ has not in general a unique result of reduction over a \texttt{SAGBI} basis $H$. The uniqueness can be obtained by modifying step (2) above. Instead of terminating when the leading term no longer can be written as a product of ${\rm lw}(H)$, we move this term to some kind of remainder (generalizing $\overline{f^H}$) and continue with the other terms.
\end{remark}

\subsection{\texttt{SAGBI} basis criterion}\label{Nordbeck1998Section3}

We start with the following proposition that gives a first method to test the \texttt{SAGBI} basis property (c.f.  \cite[Proposition 6.6.28]{KreuzerRobbiano2005}).

\begin{proposition}[{\cite[Proposition 2]{Nordbeck1998}}]\label{Nordbeck1998Proposition2}
    $H \subset S$ is a {\rm \texttt{SAGBI}} basis for the subalgebra $S$ if and only if every $f\in S$ reduces to zero over $H$.
\end{proposition}

Definition \ref{Nordbeck1998Definition3} and Proposition \ref{Nordbeck1998Proposition3} allow us to reduce the number of elements we need to consider for the \texttt{SAGBI} basis test. The definition of a $T$-polynomial is as expected from \cite[Definition 6.6.26]{KreuzerRobbiano2005} and \cite[Definition 2]{Nordbeck2002}.

\begin{definition}[{\cite[Definition 3]{Nordbeck1998}}]\label{Nordbeck1998Definition3}
Let $H$ be a subset of $\Bbbk \{x_1, \dotsc, x_n\}$. A critical pair $(m(H), m'(H))$ of $H$ is a pair of monomials $m(H), m'(H)\in \Bbbk \{x_1, \dotsc, x_n\}$ with ${\rm lw}(m(H)) = {\rm lw}(m'(H))$. If $k \in \Bbbk^{*}$ is such that ${\rm lt}(m(H)) = {\rm lt}(k m'(H))$ we define the $T$-{\em polynomial} of $(m(H), m'(H))$ as $$
T(m(H), m'(H)) = m(H) - km'(H). 
$$
\end{definition}

As in \cite[Definition 2]{Nordbeck2002}, the idea of the constant is that the leading words cancel in $T(m(H), m'(H))$, whence it follows that 
$$
{\rm lw}(T(m(H), m'(H))) \prec {\rm lw}(m(H)) = {\rm lw}(m'(H)).
$$ 

Note that 
$$
T(m(H), m'(H)) = k' T(m(H), m'(H)) \quad {\rm for\ some} \ k'\in \Bbbk, 
$$

and it can be seen that $T(m(H), m'(H))$ reduces to zero weakly (resp. strongly) if and only if $T(m(H), m'(H))$ does.

Proposition \ref{Nordbeck1998Proposition3} is the analogue of \cite[Theorem 6.6.25 (B1) and (C3)]{KreuzerRobbiano2005}, and \cite[Theorem 1]{Nordbeck2002}.

\begin{proposition}[{\cite[Proposition 3]{Nordbeck1998}}]\label{Nordbeck1998Proposition3}
$H$ is a {\rm \texttt{SAGBI}} basis if and only if the $T$-polynomials of all critical pairs of $H$ reduce to zero over $H$.
\end{proposition}

The following result establishes sufficient conditions for a set to be a \texttt{SAGBI} basis.

\begin{proposition}[{\cite[Proposition 4]{Nordbeck1998}}]\label{Nordbeck1998Proposition4}
Let $H \subset \Bbbk \{x_1, \dotsc, x_n\}$ be such that ${\rm lw}(h_i) \neq {\rm lw}(h_j)$ if $h_i \neq h_j$, $h_i, h_j \in H$. If either no word in ${\rm lw}(H)$ is a prefix {\rm (}proper left factor{\rm )} of some other word in ${\rm lw}(H)$, or no word in ${\rm lw}(H)$ is a suffix {\rm (}proper right factor{\rm )} of some other, then $H$ is a {\rm \texttt{SAGBI}} basis.
\end{proposition}

With the aim of reducing the number of critical pairs to be considered for the \texttt{SAGBI} basis test, Nordbeck formulated another version of Proposition \ref{Nordbeck1998Proposition3}. Briefly, the idea is that in the proof of the sufficiency of that result, it was only used that every $T(m(H), m'(H))$ either was zero, or could be written as a sum of monomials of height less than ${\rm lw}(m(H)) = {\rm lw}(m'(H))$. Nordbeck said that, sometimes without mentioning the corresponding $m(H)$ and $m'(H)$, that such a $T$-polynomial {\em admits a low representation}. Since this property of the $T(m(H), m'(H))$'s is clear if $H$ is a \texttt{SAGBI} basis (by reduction), then $H$ is a \texttt{SAGBI} basis if and only if every $T$-polynomial of $H$ admits a low representation \cite[Proposition 5]{Nordbeck1998}.

\begin{proposition}[{\cite[Proposition 6]{Nordbeck1998}}]\label{Nordbeck1998Proposition6}
Let $(m(H), m'(H))$ be a critical pair of $H$, and assume that there are factorizations $$
m(H) = m_1(H) m_2(H)\quad {\rm and} \quad m'(H) = m_1'(H) m_2'(H)$$ where $m_1(H)$, $m_2(H), m_1'(H), m_2'(H)$ are monomials in $\Bbbk \{H\}$ with ${\rm lw}(m_i(H)) = {\rm lw}(m_i'(H)),\ i = 1, 2$. If we have that $T(m_i(H), m_i'(H))$ admits a low representation for $i = 1, 2$, then also $T(m(H), m'(H))$ admits a low representation.
\end{proposition}

Of course, by induction we have that Proposition \ref{Nordbeck1998Proposition6} holds for every number of factors, i.e. that for any factorizations 
$$
m = m_1 \dotsb m_t \quad {\rm and} \quad m' = m_1' \dotsb m_t'
$$ 

with ${\rm lw}(m_i(H)) = {\rm lw} (m_i'(H))$ for all $i$, then the only critical pairs necessary for the \texttt{SAGBI} test are those which cannot be factored in the sense above. Notice that this is the case only for the critical pairs of the form $(h, h'),\ h, h'\in H$ with ${\rm lw}(h) = {\rm lw} (h')$, and pairs $(m, m')$, $m = h_{i_1}\dotsb h_{i_s},\ m' = h_{i_1}' \dotsb h_{i_t}'$, $s$ or (and) $t \succ 1$, where the leading words of the factors overlap. Nordbeck called a critical pair with leading words of this form an {\em overlapping pair}. With the aim of finding all such pairs we can proceed in the following way.

\begin{idea}[{\cite[p. 143]{Nordbeck1998}}]
\textquotedblleft An overlapping pair must begin with two elements of $H$ where one of the leading words is a prefix of the other. If these words are $u_1$ and $u_2 = u_1v$ we get the {\em overlap} $v$. We must then find all possibilities to continue on $v$. This could be with a leading word equal to $v$ (in which case we have obtained a critical pair), with a prefix of $v$, or finally with a word of which $v$ is a prefix. In the two latter cases we get new overlaps that might be continued. If we get an overlap we have obtained before (starting with $u_1, u_2$ or a previous pair) we need of course not continue with this; we have already examined how it can be be continued. As mentioned before, we get critical pairs whenever a leading word fits on an overlap.\textquotedblright

\textquotedblleft If $H$ is a finite set, then the process above is algorithmic. This rests on the fact that there only can be a finite number of different overlaps shorter than a given length. However, there may be an infinite number of overlapping pairs. This is the case if (and only if) the pair of leading words of some overlapping pair contains a segment beginning and ending with the same overlap. We will call such a segment a {\em loop}.\textquotedblright
\end{idea}

The following is one of the main problems in \texttt{SAGBI} theory.

\begin{problem}[{\cite[p. 144]{Nordbeck1998}}]\label{Nordbeck1998Problemp.144}
    For a finite set $H$, is the \texttt{SAGBI} basis test in general algorithmic?
\end{problem}

Below, we will some answers to this question (c.f. \cite[Proposition 6.1.3]{KreuzerRobbiano2005} and \cite[Theorem 6.6.29]{KreuzerRobbiano2005}).

\begin{proposition}[{\cite[Theorem 1]{Nordbeck1998}}]
    A subset $H\subseteq \Bbbk \{ x_1,\dotsc, x_n \}$ is a {\rm \texttt{SAGBI}} basis if and only if the $T$-polynomials of all the necessary critical pairs describes above reduce to zero over $H$.
\end{proposition}

\subsection{\texttt{SAGBI} basis construction}

Nordbeck \cite{Nordbeck1998} presented the \texttt{Algorithm 1} to compute \texttt{SAGBI} bases, where we know that $H_{\infty}$ is a \texttt{SAGBI} basis for the subalgebra $S$ generated by $H$ \cite[Proposition 12]{Nordbeck1998}. As Nordbeck stated, \textquotedblleft In general the algorithm will not stop, and we have seen in Section \ref{Nordbeck1998Section3} that Step (2) may not be algorithms even for a finite set $H$. We conclude that the construction algorithm is mostly of theoretical value\textquotedblright\ \cite[p. 146]{Nordbeck1998}.

\begin{algorithm}\label{Algorithm1Nordbeck}
\caption{}
\SetKwInOut{Input}{INPUT} 
\SetKwInOut{Output}{OUTPUT} 
\Input{A subset $H$ of $\Bbbk \{ x_1, \dotsc, x_n \}$}
\Output{A \texttt{SAGBI} basis for the subalgebra $S$ generated by $H$}
\textbf{INITIALIZATION:} $H_0 = H$; 

\BlankLine

Use the methods in Section \ref{Nordbeck1998Section3} to find the set $M_i$ of all necessary critical pairs $(m, m')$ of $H_i$.

\BlankLine

$H_{i+1} = H_i\bigcup \{\overline{T(m(H), m'(H))} \mid (m(H), m'(H)) \in M_i, \overline{T(m(H), m'(H))} \neq 0\}$. Here $\overline{T(m(H), m'(H))}$ denotes a result of reduction over $H_i$.

\BlankLine

If $H_{i+1} \neq H_i$ then go to Step (2) $(i + 1 \mapsto i)$.

\BlankLine

$H_{\infty} = \bigcup H_i$
\end{algorithm}

\section{\texttt{SAGBI} bases of SPBW extensions}\label{SAGBIbasesSPBW}

SPBW extensions generalize Ore extensions of injective type introduced by Ore \cite{Ore1933}, PBW extensions defined by Bell and Goodearl \cite{BellGoodearl1988}, 3-dimensional skew polynomial algebras introduced by Bell and Smith \cite{BellSmith1990}, ambiskew polynomial rings in the sense of Jordan \cite{Jordan1990}, solvable polynomial rings by Kandri-Rody and Weispfenning \cite{KandriRodyWeispfenning1990}, skew bi-quadratic algebras recently introduced by Bavula \cite{Bavula2023}, and different families of quantum algebras. For more details about SPBW extensions and other noncommutative algebras having PBW bases, see \cite{BuesoTorrecillasVerschoren, GoodearlWarfield2004, GomezTorrecillas2014, Levandovskyy2005, Li2002, Reyes2014, ReyesSuarezMomento2017, ReyesSuarez20173dimentional, Seiler2010}. 

In this section we recall some  results about these objects which are important for the rest of the paper.
\begin{definition}[{\cite[Definition 1]{GallegoLezama2010}}]\label{def.skewpbwextensions}
Let $R$ and $A$ be rings. We say that $A$ is a \textit{SPBW extension} over $R$ if the following conditions hold:
\begin{enumerate}
\item[\rm (i)] $R$ is a subring of $A$ sharing the same identity element.

\item[\rm (ii)] There exist finitely many elements $x_1,\dots ,x_n\in A$ such that $A$ is a left free $R$-module, with basis the
set of standard monomials
\begin{center}
${\rm Mon}(A):= \{x^{\alpha}:=x_1^{\alpha_1}\cdots x_n^{\alpha_n}\mid \alpha=(\alpha_1,\dots
,\alpha_n)\in \mathbb{N}^n\}$.
\end{center}
Moreover, $x^0_1\cdots x^0_n := 1 \in {\rm Mon}(A)$.

\item[\rm (iii)] For each $1\leq i\leq n$ and any $r\in R\ \backslash\ \{0\}$, there exists an element $c_{i,r}\in R\ \backslash\ \{0\}$ such that
$x_ir-c_{i,r}x_i\in R.$

\item[\rm (iv)] For $1\leq i,j\leq n$ there exists $d_{i,j}\in R\ \backslash\ \{0\}$ such that
\begin{equation}\label{sigmadefinicion2}
x_jx_i - d_{i,j}x_ix_j\in R+Rx_1+\cdots +Rx_n,
\end{equation}

or equivalently, 
\begin{equation}\label{eq.rep variables x}
x_jx_i = d_{i,j}x_ix_j+r_{0_{j,i}} + r_{1_{j,i}}x_{1} + \cdots +
r_{n_{j,i}}x_{n},
\end{equation}
where $d_{i,j},r_{0_{j,i}},  r_{1_{j,i}}, \dots, r_{n_{j,i}}\in R$, for $1\leq i,j\leq n$.

Under these conditions, we write $A=\sigma(R)\langle x_1,\dots,x_n\rangle$.
\end{enumerate}
\end{definition}
\begin{remark}[{\cite[Remark 2]{GallegoLezama2010}}]\label{notesondefsigampbw} 
Let $A=\sigma(R)\langle x_1, \dots, x_n\rangle$ be a SPBW extension over $R$.
\begin{enumerate}
\item [\rm (i)] Since ${\rm Mon}(A)$ is a left $R$-basis of $A$, the elements $c_{i,r}$ and $d_{i,j}$ in Definition \ref{def.skewpbwextensions} are unique.

\item [\rm (ii)] If $r = 0$, then $c_{i,0}= 0$. In Definition \ref{def.skewpbwextensions} (iv), $d_{i,i}=1$. This follows from 
$$
x_i^2-d_{i,i}x_i^2=s_0+s_1x_1+\cdots + s_nx_n \quad {\rm with} \ s_j\in R,
$$

which implies that $1-d_{i,i}=0=s_j$ for $0\leq j\leq n$.

\item [\rm (iii)] Let $i<j$. By (\ref{sigmadefinicion2}) there exist elements
$d_{j,i}, d_{i,j}\in R$ such that 
$$
x_ix_j - d_{j,i}x_jx_i\in R+Rx_1+\cdots +Rx_n
$$ 

and
$$
x_jx_i-d_{i,j}x_ix_j\in R+Rx_1+\cdots+Rx_n,
$$

whence $1 = d_{j,i}d_{i,j}$, that is, for each
$1\leq i<j\leq n$, $d_{i,j}$ has a left inverse and $d_{j,i}$ has a right inverse. In general, the elements $d_{i,j}$ are not two-sided invertible. For instance, 
$$
x_1x_2 = d_{2,1}x_2x_1+p = d_{21}(d_{1,2}x_1x_2+q) + p,
$$

where $p,q\in R+Rx_1+\dotsb + Rx_n$, so that $1 = d_{2,1}d_{1,2}$ since $x_1x_2$ is a basic element of ${\rm Mon}(A)$. Now, 
$$
x_2x_1 = d_{1,2}x_1x_2+q = d_{1,2}(d_{2,1}x_2x_1+p) + q
$$

but we cannot conclude that $d_{1,2}d_{2,1}=1$
because $x_2x_1$ is not a basic element of ${\rm Mon}(A)$.

\item[\rm (iv)] Every element $f\in A\ \backslash\ \{0\}$ has a unique representation as a linear combination of monomials $$
f = \sum\limits_{i=1}^{t} r_iX_i \quad {\rm with} \ r_i\in R\ \backslash\ \{0\}, \ X_i\in {\rm Mon}(A) \ {\rm for} \ 1\leq i\leq t.
$$
\end{enumerate}
\end{remark}

Proposition \ref{sigmadefinition} shows the relationship between SPBW extensions and Ore extensions or skew polynomial rings.
\begin{proposition}[{\cite[Proposition 3]{GallegoLezama2010}}]\label{sigmadefinition}
Let $A=\sigma(R)\langle x_1, \dots, x_n\rangle$ be a SPBW extension over $R$. For each $1\leq i\leq n$, there exist an injective endomorphism $\sigma_i:R\rightarrow
R$ and a $\sigma_i$-derivation $\delta_i:R\rightarrow R$ such that $x_ir=\sigma_i(r)x_i+\delta_i(r), \ r \in R.$
\end{proposition}

\begin{definition}[{\cite[Definitions 6 and 11]{GallegoLezama2010}}]\label{definitioncoefficients}
Let $A = \sigma(R)\langle x_1,\dotsc, x_n\rangle$ be a SPBW extension over $R$.
\begin{enumerate}
\item[(i)] For $\alpha=\left(\alpha_{1}, \dotsc, \alpha_{n}\right) \in \mathbb{N}^{n}, \sigma^{\alpha}:=\sigma_{1}^{\alpha_{1}} \circ \dotsb  \circ \sigma_{n}^{\alpha_{n}}$, where $\circ$ denotes the classical composition of functions, $|\alpha|:=\alpha_{1}+\cdots+\alpha_{n} .$ If $\beta=\left(\beta_{1}, \dotsc, \beta_{n}\right) \in \mathbb{N}^{n},$ then $\alpha+\beta:=\left(\alpha_{1}+\beta_{1}, \dotsc, \alpha_{n}+\beta_{n}\right)$.
\item[(ii)] For $X=x^{\alpha} \in \operatorname{Mon}(A),$ $\exp (X):=\alpha$ and $\deg(X):=|\alpha|$.

\item[(iii)] Since every element $f\in A$ can be written uniquely as $f = \sum\limits_{i=1}^{t} r_iX_i$ (Remark \ref{notesondefsigampbw} (iv)), let $\deg(f):=\max \left\{\deg\left(X_{i}\right)\right\}_{i=1}^{t}$.

\item[(iv)] Let $\preceq$ be a total order defined on $\operatorname{Mon}(A).$ If $x^{\alpha} \preceq x^{\beta}$ but $x^{\alpha} \neq x^{\beta}$, we will write $x^{\alpha} \prec x^{\beta}.$ $x^{\beta} \preceq x^{\alpha}$ means $x^{\alpha} \succeq x^{\beta}.$  Each element $f \in A\ \backslash\ \{0\}$ can be represented in a unique way as $f=c_1 x^{\alpha_1}+\cdots+c_t x^{\alpha_t}$, with $c_i \in R\ \backslash\ \{0\}, 1 \leq i \leq t$, and $x^{\alpha_1} \succ \cdots \succ x^{\alpha_t}$. We say that $x^{\alpha_1}$ is the \textit{leading monomial} of $f$ and we write $\operatorname{lm}(f):=x^{\alpha_1} ; c_1$ is the \textit{leading coefficient} of $f,\ {\rm lc}(f):=c_1$, and $c_1 x^{\alpha_1}$ is the \textit{leading term} of $f$ denoted by ${\rm lt}(f):=c_1 x^{\alpha_1}$. If $f=0$, we define $\operatorname{lm}(0):=0, \operatorname{lc}(0):=0, \operatorname{lt}(0):=0$.

We say that $\preceq$ is a \textit{monomial order} (also called {\em admissible order}) on $\operatorname{Mon}(A)$ if the following conditions hold:
\begin{enumerate}
\item[\rm (a)] For every $x^\alpha, x^\beta, x^\gamma, x^\lambda \in \operatorname{Mon}(A)$, the relation $x^\alpha \preceq x^\beta$ implies that $\operatorname{lm}(x^\gamma x^\alpha  x^\lambda) \preceq \operatorname{lm}(x^\gamma x^\beta x^\lambda)$.

\item[\rm (b)] $1 \preceq x^\alpha$ for every $x^\alpha \in \operatorname{Mon}(A)$.

\item[\rm (c)] $\preceq$ is degree compatible, i.e.  $|\alpha| \leq |\beta|$ implies $x^\beta \preceq x^\alpha.$
\end{enumerate}
\end{enumerate}
\end{definition}

The condition (iv)(c) of the previous definition is needed in the proof that every monomial order on ${\rm Mon}(A)$ is a well order. Thus, there are not infinite decreasing chains in ${\rm Mon}(A)$ \cite[Proposition 12]{GallegoLezama2010}. Nevertheless, this hypothesis is not really needed to get a well ordering if a more elaborated argument, based upon Dickson's Lemma, is developed (e.g. \cite[Theorem 4.6.2]{BeckerWeispfenning1993}). 

Let us see some classical examples of monomial orderings (for more details, see \cite[Chapter 2, Section 2]{CoxLittleOshea2015}).

\begin{example}\label{Monomialordersexamples}
\item [\rm (i)] (\texttt{Lexicographic Order}) Let $\alpha = (\alpha_1, \dotsc, \alpha_n)$ and $\beta = (\beta_1, \dotsc, \beta_n)$ be in $\mathbb{Z}_{\ge 0}^{n}$. We say $\alpha \succ \beta$ if the leftmost non-zero entry of the vector difference $\alpha - \beta \in \mathbb{Z}_{\ge 0}^{n}$ is positive. We will write $x^{\alpha} \succ_{\texttt{lex}} x^{\beta}$ if $\alpha \succ_{\texttt{lex}} \beta$.

The indeterminates $x_1, \dotsc, x_n$ are ordered in the usual way by the \texttt{lex} ordering
\[
(1, 0, \dotsc, 0) \succ_{\texttt{lex}} (0, 1, 0, \dotsc, 0) \succ_{\texttt{lex}} \dotsb \succ_{\texttt{lex}} (0, \dotsc, 0, 1),
\]
so $x_1 \succ_{\texttt{lex}} x_2 \succ_{\texttt{lex}} \dotsb \succ_{\texttt{lex}} x_n$.

If we work with polynomials in two or three indeterminates, we will write $x, y, z$ rather than $x_1, x_2, x_3$.

\item [\rm (ii)] (\texttt{Degree Lex Order}) Let $\alpha, \beta \in \mathbb{Z}_{n\ge 0}^n$. We say $\alpha \succ_{\texttt{deglex}} \beta$ if
\[
|\alpha| = \sum_{i = 1}^{n} \alpha_i > |\beta| = \sum_{i=1}^{n} \beta_i, \quad {\rm or}\quad  |\alpha| = |\beta|\ \ {\rm and}\ \ \alpha\succ_{\rm lex}\ \beta.
\]
We will write $x^{\alpha} \succ_{\texttt{deglex}} x^{\beta}$ if $\alpha \succ_{\texttt{deglex}} \beta$.

\item [\rm (iii)] (\texttt{Degree Reverse Lex Order}) Let $\alpha, \beta \in \mathbb{Z}_{n\ge 0}^n$. We say $\alpha \succ_{\texttt{degrevlex}} \beta$ if
\begin{align*}
 |\alpha| = &\ \sum_{i = 1}^{n} \alpha_i > |\beta| = \sum_{i=1}^{n} \beta_i, \quad {\rm or} \\
 |\alpha| = &\ |\beta|\ \ {\rm and\ the\ rightmost\ non}-{\rm zero\ entry\ of}\ \ \alpha - \beta \in \mathbb{Z}^n\ {\rm is\ negative}.
\end{align*}

We will write $x^{\alpha} \succ_{\texttt{degrevlex}} x^{\beta}$ if $\alpha \succ_{\texttt{degrevlex}} \beta$.
\end{example}

We start by presenting some terminology used in this section.

\subsection{Basic definitions}\label{SPBWSAGBIBasicDefinitions}

\begin{definition}
From now on, we consider $A=\sigma(R)\left\langle x_{1}, \dotsc, x_{n}\right\rangle$ as a SPBW extension over a $\Bbbk$-algebra $R$ (that is, $\sigma_{i}(k)=k$ and $\delta_{i}(k)=0$, for every $k \in \Bbbk$, and $1 \leq i \leq n$, as in Proposition \ref{sigmadefinition}), and $F$ a finite set of non-zero elements of $A$. Let $\preceq$ be a monomial ordering on $A$ in the sense of Definition \ref{definitioncoefficients}(iv).
\begin{enumerate}
\item[\rm (i)] The notation $\Bbbk\langle F\rangle_{A}$ means {\em the subalgebra of} $A$ {\em generated by} $F$. The elements of $\Bbbk\langle F\rangle_{A}$ are precisely the polynomials in the set of formal intederminates $F$, viewed as elements of $A$. The elements of $\Bbbk \subset \Bbbk\langle F\rangle_{A}$ correspond to the constant polynomials. 

\item[\rm (ii)] {\em By an} $F$-{\em monomial} {\em we mean a finite product of elements from} $F$ that will usually be written as $m(F)$ (the \textquotedblleft empty\textquotedblright\ monomial $1$ is considered). When we speak of the leading monomial, leading coefficient and leading term of an element in $\Bbbk\langle F\rangle_{A}$, we will always mean the leading monomial, leading coefficient and leading term, respectively, of the element viewed as an element of $A$, relative to the monomial ordering $\preceq$ in $A$.

\item[\rm (iii)] $\operatorname{lm}(F):=\left\{\operatorname{lm}(f) \mid f \in F\right\}$. Since orders are preserved after multiplication, if $m(F) = f_{i_ 1} f_{i_ 2} \cdots f_{i_ t}$, with $f_{i_j}\in F$, then $$\operatorname{lt}(m(F))= \operatorname{lt}\left(\operatorname{lt}\left(f_{i_ 1}\right) \operatorname{lt}\left(f_{i_ 2}\right) \cdots \operatorname{lt}\left(f_{i_ t}\right)\right)$$ and $$ \operatorname{lm}(m(F))= \operatorname{lm}\left(\operatorname{lm}\left(f_{i_ 1}\right) \operatorname{lm}\left(f_{i_ 2}\right) \cdots \operatorname{lm}\left(f_{i_ t}\right)\right).$$
\end{enumerate}
\end{definition}

\begin{example}\label{mondiffusion}
Consider the diffusion algebra  $\sigma(\mathbb{Q}[x_1,x_2, x_3])\langle D_1, D_2, D_3\rangle$ (see \cite[Section 2.4]{Fajardoetal2020} for more details), and let 
$$
F :=\left\{f_1 :=x_1x_2D_1D_2+x_3D_1D_3, f_2:=D_2^2D_3^2\right\}
$$ 

be a subset of the algebra with the monomial ordering \texttt{deglex}, and $D_1\succ D_2\succ D_3$. An example of an $F$-monomial $m(F) \in \Bbbk\langle F\rangle_A$ is given by
	\begin{align*}
	m(F) = &\ f_1f_2 = \left(x_1x_2D_1D_2+x_3D_1D_3\right)D_2^2D_3^2\\ = &\ x_1x_2D_1D_2^3D_3^2+x_3D_1\left(D_3D_2\right) \left(D_2D_3^2\right)\\ = &\ x_1x_2D_1D_2^3D_3^2+x_3D_1 \left(D_2D_3+x_3D_2-x_2D_3\right) \left(D_2D_3^2\right)\\ = &\ x_1x_2D_1D_2^3D_3^2+x_3D_1D_2\left(D_3D_2\right)D_3^2+x_3^2D_1D_2^2D_3^2-x_2x_3D_1\left(D_3D_2\right)D_3^2\\ = &\ 	x_1x_2D_1D_2^3D_3^2+x_3D_1D_2\left(D_2D_3+x_3D_2-x_2D_3\right) D_3^2+x_3^2D_1D_2^2D_3^2\\&-x_2x_3D_1\left(D_2D_3+x_3D_2-x_2D_3\right)D_3^2\\ = &\ 	x_1x_2D_1D_2^3D_3^2+x_2^2x_3D_1D_3^3-x_2x_3^2D_1D_2D_3^2\\ & -2x_2x_3D_1D_2D_3^3+2x_3^2D_1D_2^2D_3^2+x_3D_1D_2^2D_3^3.
	\end{align*}
\end{example}

\begin{example}

Following Havli\v{c}ek et al. \cite[p. 79]{HavlicekKlimykPosta2000}, the $ \mathbb{C}$-algebra $U_q'(\mathfrak{so}_3)$ over the field of complex numbers $\mathbb{C}$ is generated by the indeterminates $I_1, I_2$, and $I_3$, subject to the relations given by
\begin{equation*}
  I_2I_1 - qI_1I_2 = -q^{\frac{1}{2}}I_3,\quad  I_3I_1 - q^{-1}I_1I_3 = q^{-\frac{1}{2}}I_2,\quad     I_3I_2 - qI_2I_3 = -q^{\frac{1}{2}}I_1,
\end{equation*}

where $q\in \mathbb{C}^{*}$. By using \cite[Theorem 1.14]{ReyesSuarezMomento2017}, it can be shown that $U_q'(\mathfrak{so}_3) \cong \sigma(\mathbb{C}) \langle I_1, I_2, I_3\rangle$.

Consider the subset 
$$
F =\left\{f_1:=I_1I_2+I_3, f_2:=I_1I_3+I_2\right\}
$$ 
of $U'_q(\mathfrak{so}_3)$ with the monomial ordering \texttt{deglex} and $I_1\succ I_2\succ I_3$. Then: 
\begin{align*}m(F) = &\ f_1f_2 \\
	= &\ \left(I_1I_2+I_3\right) \left(I_1I_3+I_2\right)\\
	= &\ I_1\left(I_2I_1\right) I_3+I_1I_2^2 + \left(I_3I_1\right) I_3 + I_3I_2\\ = & \ I_1\left(qI_1I_2-q^{1/2}I_3\right) I_3+I_1I_2^2+I_3\left(q^{-1}I_1I_3+q^{-1/2}I_2\right) + qI_2I_3-q^{-1/2}I_1\\= & \ qI_1^2I_2I_3-q^{-1/2}I_1I_3^2+I_1I_2^2+q^{-1} \left(q^{-1}I_1I_3+q^{-1/2}I_2\right) I_3\\ & +q^{-1/2}\left(qI_2I_3-q^{1/2}I_1\right) + qI_2I_3-q^{-1/2}I_1\\ = &\  qI_1^2I_2I_3+I_1I_2^2 + \left(q^{-2}-q^{-1/2}\right) I_1I_3^2 + \left(q^{-3/2}+q^{1/2}\right) I_2I_3\\ & + qI_1I_3 - \left(1+q^{-1/2}\right) I_1	
\end{align*}
 
Notice that 
\begin{align*}
	\operatorname{lm}(f_1)\operatorname{lm}(f_2) = &\ (I_1I_2)(I_1I_3)=I_1(qI_1I_2-q^{1/2}I_3)I_3=qI_1^2I_2I_3-q^{1/2}I_1I_3^2, \\
	\operatorname{lm}\left(\operatorname{lm}\left(f_{ 1}\right) \operatorname{lm}\left(f_{2}\right)\right) = &\ qI_1^2I_2I_3=\operatorname{lm}(m(F)).
\end{align*}
\end{example}

\begin{example}
Let $U(\mathfrak{osp}(1,2))$ be the universal enveloping algebra of the Lie superalgebra $\mathfrak{osp}(1,2)$. This is generated by the indeterminates $x,y,z$ subject to the relations 
$$
yz - zy = z,\ zx + xz = y \quad  {\rm and} \quad xy - yx = x.
$$

From \cite[p. 1215]{LezamaReyes2014} we know that $U(\mathfrak{osp}(1,2)) \cong \sigma(\Bbbk)\langle x,y,z \rangle$. Consider the subset 
$$
F =\{f_1=x^2y, f_2=xy+z\}
$$ 

of $U(\mathfrak{osp}(1,2))$ with the monomial ordering \texttt{deglex} and $x\succ y\succ z.$ An example of an $F$-monomial $m(F) \in \Bbbk\langle F\rangle_A$ is given by
\begin{align*}
m(F) = &\ f_2f_1 = \left(xy+z\right)x^2y\\ = &\ x(yx)xy+(zx)xy\\ = &\ x(xy-x)xy+(-xz+y)xy\\ = &\ x^2(yx)y-x^3y-x(zx)y+(yx)y\\ = &\ x^2(xy-x)y-x^3y-x(-xz+y)y+(xy-x)y\\ = &\ x^3y^2-2x^3y+x^2(zy)-xy\\ = &\ x^3y^2-2x^3y+x^2(yz-z)-xy\\ = &\ x^3y^2-2x^3y+x^2yz-x^2z-xy.
\end{align*}
	
Note that 
\begin{align*}
   \operatorname{lm}(f_1)\operatorname{lm}(f_2)= &\ (xy)(x^2y)=x(xy-x)xy = x^2(yx)y-x^3y\\
   = &\ x^2(xy-x)-x^3y=x^3y^2-2x^3y,  
\end{align*}

and 
$$\operatorname{lm}(m(F))=\operatorname{lm}\left(\operatorname{lm}\left(f_{ 1}\right) \operatorname{lm}\left(f_{2}\right)\right)=x^3y^2. $$
\end{example}

\begin{example}
In Example \ref{mondiffusion}, 
$$
\operatorname{lm}(f_1)\operatorname{lm}(f_2)=(D_1D_2)(D_2^2D_3^2)=D_1D_2^3D_3^2
$$

and $\operatorname{lm}(m(F))$  is equal to $\operatorname{lm}\left(\operatorname{lm}\left(f_{ 1}\right) \operatorname{lm}\left(f_{2}\right)\right)= D_1D_2^3D_3^2$.
\end{example}

\subsection{Reduction and Algorithm for the \texttt{SAGBI} normal form}

We present the notion of {\em reduction} which is necessary in the characterization of \texttt{SAGBI} bases (c.f. \cite[Definition 6.6.16]{KreuzerRobbiano2005}, \ref{Nordbeck1998reduction} and  \cite[Definition 3]{Khanetal2019}).

\begin{reduction}\label{normalform}
Let $A = \sigma(R)\langle x_1,\dotsc, x_n\rangle$ be a SPBW extension over a $\Bbbk$-algebra $R$, $F$ a finite set of non-zero elements of $A$, and $s, s_0$ elements of $A$. We say that $s_0$ is a \textit{one-step $s$-reduction of} $s$ with respect to $F$ if there exist an $F$-monomial $m(F)$ and $k \in \Bbbk$ such that the following two conditions hold:
\begin{enumerate}
 \item [\rm (i)] $k\operatorname{lt}(m(F))=\operatorname{lt}(s)$, and 
 \item [\rm (ii)] $s_0=s-km(F).$
\end{enumerate}
\end{reduction}

If the first condition of Reduction \ref{normalform} fails, then the $s$-reduction of $s$ with respect to $F$ is $s.$ When we apply the one-step $s$-reduction process iteratively, we obtain a special form of $s$ with respect to $F$ (which cannot be $s$-reduced further with respect to $F$), called \texttt{SAGBI} normal form (c.f. \cite[Proposition 3.2]{LezamaMarin2009}). In this case, we write $\mathrm{s}_{0}:=\texttt{SNF}(\mathrm{s} | \mathrm{F})$ (c.f. \cite[Proposition 3.2]{LezamaMarin2009} and Reduction \cite[Definition 3]{Khanetal2019}). \texttt{Algorithm} \ref{Algorithm1} allows us to compute the \texttt{SNF}. 

\begin{algorithm}\label{Algorithm1}
\caption{}
\SetKwInOut{Input}{INPUT} 
\SetKwInOut{Output}{OUTPUT} 
\Input{A fixed monomial ordering $\succeq$ on $\operatorname{Mon}(A)$, $F \subseteq A$ and $s \in A.$}
\Output{$h \in A$, the \texttt{SAGBI} normal form}
\textbf{INITIALIZATION:} $s_{0}:=s$;
$F_{s_{0}}:=\left\{k m(F) \mid k \in \Bbbk \text { and } k\operatorname{lt}( m(F))=\operatorname{lt}\left(s_{0}\right)\right\}$;
\BlankLine

\While{$s_{0} \neq 0$ {\rm and} $F_{s_{0}} \neq \emptyset$}{
   choose $km(F) \in F_{s_0}$;
$s_0 := s_0 -km(F)$;
$F_{s_0} := \{km(F)\mid k \in \Bbbk$ and $k\operatorname{lt}(m(F)) = \operatorname{lt}(s_0)\}$\;
}
\textbf{return} $s_0$;
\end{algorithm}

Notice that for different choices of $km(F)$, the output of \texttt{SNF} may also be different. Examples \ref{examplered}, \ref{examplered2} and \ref{examplered3} illustrate this situation.

All algebras appearing in the following examples are SPBW extensions. Details can be found in Su\'arez \cite[Chapter 1]{SuarezPhDThesis2023}.

\begin{example}\label{examplered}
The {\em Jordan plane} $\mathcal{J}$ introduced by Jordan \cite{Jordan2001} is the free $\Bbbk$-algebra generated by the indeterminates $x,y$ subject to the relation $yx = xy + x^2$. Let $\Bbbk\langle F\rangle_{\mathcal{J}}$ be the subalgebra generated by the set $F:=\{q_1, q_2, q_3\}$, where 
$$
q_1:= x^2, q_2:= y, q_3:= xy+y \quad {\rm and} \quad g = x^3y+y\in \mathcal{J}.
$$

For the computation of $\texttt{SNF}(g | F)$, we use \texttt{Algorithm} \ref{Algorithm1}. 

Let $g=x^3y+y:=s_0$. If $F_{s_0}=\{q_1q_3,q_3q_1\}$, then we have two possibilities to choose $km(F) \in F_{s_0}$:
\begin{enumerate}
\item [\rm (i)] If $km(F):= q_1q_3$, then $s_0:=-x^2y+y$ and $F_{s_0}:=\{-q_1q_2,-q_2q_1\}$. Now, we take $km(F)=-q_2q_1$, which implies that $s_0:=y\in F,$ whence $\texttt{SNF}(g | F)=0$.
\item [\rm (ii)] If $km(F):= q_3q_1$, then $s_0:=-2x^4-x^2y-2x^3+y$ and $F_{s_0}:=\{-2(q_1)^2\}$, and so $km(F)=-2(q_1)^2$, whence $s_0:=-x^2-2x^3+y$, which implies that $F_{s_0}:=\{-q_1q_2,-q_2q_1\}.$ In this way, $km(F)=-q_1q_2,$ and we obtain $s_0:=-2x^3+y$ and $F_{s_0}=\emptyset.$ Thus, $\texttt{SNF}(g | F)=-2x^3+y.$
\end{enumerate}
\end{example}

\begin{example}
Zhedanov \cite[Section 1]{Zhedanov1991} introduced the {\em Askey-Wilson algebra} $AW(3)$ as the $\mathbb{R}$-algebra (as usual, $\mathbb{R}$ denotes the field of real numbers) generated by three operators $K_0, K_1$, and $K_2$, that satisfy the commutation relations 
\begin{align*}
	[K_0, K_1]_{\omega} = &\ K_2, \\
	[K_2, K_0]_{\omega} = &\ BK_0 + C_1K_1 + D_1, \quad {\rm and}\\
	[K_1, K_2]_{\omega} = &\ BK_1 + C_0K_0 + D_0,
\end{align*}

where $B, C_0, C_1, D_0$, and $D_1$ are elements of $\mathbb{R}$ that represent the structure constants of the algebra, and the $q$-commutator $[ - , -]_{\omega}$ is given by $[\square, \triangle]_{\omega}:= e^{\omega}\square \triangle - e^{- \omega}\triangle \square$, where $\omega\in \mathbb{R}$. Notice that in the limit $\omega \to 0$, the algebra AW(3) becomes an ordinary Lie algebra with three generators ($D_0$ and $D_1$ are included among the structure constants of the algebra in order to take into account algebras of Heisenberg-Weyl type). The relations above defining the algebra can be written as 
\begin{align*}
    e^{\omega}K_0K_1 - e^{-\omega}K_1K_0 = &\ K_2,\\
    e^{\omega} K_2K_0 - e^{-\omega}K_0 K_2 = &\ BK_0 + C_1K_1 + D_1,\\
    e^{\omega}K_1K_2 - e^{-\omega}K_2K_1 = &\ BK_1 + C_0K_0 + D_0.
\end{align*}

Consider $AW(3)$  with the monomial ordering \texttt{deglex} and $K_0\succ K_1\succ K_2.$ Let $\Bbbk\langle F\rangle_{AW(3)}$ be the subalgebra generated by the set $F:=\{q_1, q_2, q_3\}$, where 
$$
q_1:=K_0 K_1+K_2,\ q_2:= K_0,\ q_3:= K_0 K_2, \quad {\rm and} \quad g =K_0^2 K_1 K_2+K_2\in AW(3).
$$

Let us find $\texttt{SNF}(g | F)$ by using \texttt{Algorithm} \ref{Algorithm1}. 

Let $s_0 := g = K_0^2 K_1 K_2+K_2$. If $F_{s_0}=\{e^{-2\omega}q_1q_3,q_3q_1\}$, then we have two options to choose $km(F) \in F_{s_0}$:
\begin{enumerate}
\item [\rm (i)] If
\begin{align*}
km(F):= &\ e^{-2\omega}q_1q_3 \\
= &\ e^{-2\omega}(e^{2 \omega} K_0^2 K_1 K_2+\left(e^{-2 \omega} - e^\omega\right) K_0 K_2^2+e^{-\omega} B K_0 K_2+e^{-\omega} C_1 K_1 K_2\\
 &+e^{-\omega} D_1 K_2),
\end{align*}

then 
\[
s_0:=\left(e^{-\omega}-e^{-4 \omega}\right) K_0 K_2^2-e^{-3 \omega} B K_0 K_2-e^{-3 \omega} C_1 K_1 K_2+\left(1-e^{-3 w} D_1\right) K_2, 
\]

and $F_{s_0}=\emptyset$. Thus, 
\[
\texttt{SNF}(g | F)=\left(e^{-\omega}-e^{-4 \omega}\right) K_0 K_2^2-e^{-3 \omega} B K_0 K_2-e^{-3 \omega} C_1 K_1 K_2+\left(1-e^{-3 w} D_1\right) K_2.
\]

\item [\rm (ii)] If
\begin{align*}
km(F):= &\ q_3q_1 \\
= &\ K_0^2 K_1 K_2-e^{-\omega} C_0 K_0^3+e^{-\omega} C_1 K_0 K_1^2+K_0 K_2^2-e^{-\omega} D_0 K_0^2\\
 &+e^{-\omega} D_1 K_0 K_1,
\end{align*}

then
\[
s_0:=e^{-\omega} C_0 K_0^3-e^{-\omega} C_1 K_0 K_1^2-K_0 K_2^2+e^{-\omega} D_0 K_0^2-e^{-\omega} D_1 K_0 K_1+K_2,
\]

and $F_{s_0}:=\{e^{-\omega}C_0q_2^3\}$. In this way, 
$$
km(F)=e^{-\omega}C_0q_2^3, 
$$ 

whence
\[
s_0:=-e^{-\omega} C_1 K_0 K_1^2-K_0 K_2^2+e^{-\omega} D_0 K_0^2-e^{-\omega} D_1 K_0 K_1+K_2,
\]

which implies that $F_{s_0} = \emptyset$. We conclude that
\[
\texttt{SNF}(g | F)=-e^{-\omega} C_1 K_0 K_1^2-K_0 K_2^2+e^{-\omega} D_0 K_0^2-e^{-\omega} D_1 K_0 K_1+K_2.
\] 
\end{enumerate}
\end{example}

\begin{example}\label{examplered3}
From \cite[p. 41]{GoodearlWarfield2004}, for $q$ an element of $\Bbbk$ with $q\neq \pm 1$, the {\em quantized enveloping algebra} of $\mathfrak{sl}_2(\Bbbk)$ corresponding to the choice of $q$ is the $\Bbbk$-algebra $U_q(\mathfrak{sl}_2(\Bbbk))$ presented by the generators $E, F, K, K^{-1}$ and the relations 
\begin{align*}
	KK^{-1} = &\ K^{-1}K = 1,\\ 
	EF - FE = &\ \frac{K-K^{-1}}{q-q^{-1}}, \\ KE = &\ q^{2}EK, \ {\rm and}\\ KF = &\ q^{-2}FK.
\end{align*}

Consider $U_q(\mathfrak{sl}_2(\Bbbk))$ with the monomial ordering \texttt{deglex}, and $E\succ F\succ K\succ K^{-1}.$ Let $\Bbbk\langle \mathcal{F}\rangle_{U_q(\mathfrak{sl}_2(\Bbbk))}$ be the subalgebra generated by the set $\mathcal{F}:=\{q_1, q_2, q_3\}$, with 
$$
q_1:=E K-F K^{-1},\ q_2:= EF,\ q_3:= F K^{-1} \ {\rm and} \ g = E^2 F K+E F K^2\in U_q(\mathfrak{sl}_2(\Bbbk)).
$$

With the aim of finding $\texttt{SNF}(g | \mathcal{F})$, let $g=K_0^2 K_1 K_2+K_2:=s_0$. If $F_{s_0}=\{q_1q_2,q_2q_1\}$, then we obtain that:
\begin{enumerate}
\item [\rm (i)] If 
$$
km(\mathcal{F}) := q_1q_2=E^2 F K-E F^2 K^{-1}+\frac{F-F K^{-2}}{q-q^{-1}}, 
$$

then 
$$
s_0:=E F^2 K^{-1}+E F K^2-\frac{F-F K^{-2}}{q-q^{-1}}
$$ 

and $F_{s_0}=\{q_2q_3,q_3q_2\}.$ 

Now, if 
$$
km(\mathcal{F})=q_2q_3=E F^2 K^{-1},
$$

we get that 
$$
s_0:=E F K^2-\frac{F-F K^{-2}}{q-q^{-1}},
$$

which implies $F_{s_0}=\emptyset.$ Thus, 
$$
\texttt{SNF}(g |\mathcal{F} )= E F K^2-\frac{F-F K^{-2}}{q-q^{-1}}.
$$ 

On the other hand, if $$
km(\mathcal{F})=q_3q_2=E F^2 K^{-1}-\frac{F-F K^{-2}}{q-q^{-1}}, 
$$

it follows that $s_0:=E F K^2$, and so $F_{s_0}=\emptyset$. Thus, 
$$
\texttt{SNF}(g |\mathcal{F} )= E F K^2.
$$

\item [\rm (ii)] Let 
$$
km(\mathcal{F}):=q_2q_1=E^2 F K-E F^2 K^{-1}-\frac{EK^2-E}{q-q^{-1}}.
$$

We have that 
$$
s_0:=E F^2 K^{-1}+E F K^2+\frac{EK^2-E}{q-q^{-1}}
$$ and $F_{s_0}=\{q_2q_3,q_3q_2\}$. If 
$$
km(\mathcal{F})=q_2q_3=E F^2 K^{-1}, 
$$ 

we obtain 
$$
s_0:=E F K^2+\frac{EK^2-E}{q-q^{-1}},
$$

whence $F_{s_0}=\emptyset$. This implies that 
$$
\texttt{SNF}(g |\mathcal{F} )=E F K^2-\frac{EK^2-E}{q-q^{-1}}.
$$ 

Now, considering $$
km(\mathcal{F})=q_3q_2=E F^2 K^{-1}-\frac{F-F K^{-2}}{q-q^{-1}}, $$ 

it follows that 
$$
s_0:=E F K^2+\frac{EK^2-E}{q-q^{-1}}+\frac{F-F K^{-2}}{q-q^{-1}}, 
$$

and so $F_{s_0}=\emptyset$. Hence, 
$$\texttt{SNF}(g |\mathcal{F} )= E F K^2+\frac{EK^2-E}{q-q^{-1}}+\frac{F-F K^{-2}}{q-q^{-1}}.
$$
\end{enumerate}
\end{example}

If $S$ is a subalgebra of a SPBW extension $A$ and $F$ is a subset of $S$, then our interest lies in the case when \texttt{SAGBI} normal form is zero. If there is at least one choice of $F$-monomials such that $s_{0}=0,$ then we say $s$ \textit{reduces weakly over} $F,$ and \textit{reduces strongly} if all possible choices give $s_{0}=0$. This is the same as the case of $G$-algebras.

\begin{example}
The reduction process done in Example \ref{examplered} shows us that $g=x^3y+y\in \mathcal{J}$ reduces weakly over $F.$
\end{example}

\begin{example}\label{examplered2}
Consider the 3-dimensional skew polynomial algebra $A$ defined by Bell and Smith \cite{BellSmith1990} (see also \cite{ReyesSuarez20173dimentional}) which is generated over $\Bbbk$ by the indeterminates $x,y,z$ restricted to relations
$$
yz- zy=z,\ zx+xz=0 \quad {\rm and} \quad xy-yx=x. 
$$ 

Let $\Bbbk\langle F\rangle_{A}$ be the subalgebra generated by the set $F:=\{q_1, q_2, q_3, q_4\}$, where 
$$
q_1:= x^2, \ q_2:= z, \ q_3=xy+z, \ q_4=x^3
$$ 

with the monomial ordering \texttt{deglex} given by $x\succ y\succ z$. Consider the element $s=x^3y+z\in A$. 
Let $s=x^3y+z:=s_0$. If $F_{s_0}=\{q_1q_3,q_3q_1\}$, then we have two options to choose $km(F) \in F_{s_0}$:
\begin{enumerate}
\item [\rm (i)] If $km(F):= q_1q_3$, then $s_0:=-x^2z+z$ and $F_{s_0}:=\{-q_1q_2=-q_2q_1\}.$ Now, we take $km(F)=-q_2q_1$, whence $s_0:=z\in F,$ whence $\texttt{SNF}(s | F)=0$.
\item [\rm (ii)] If $km(F):= q_3q_1$, then $s_0:=2x^3+x^2z+z$ and $F_{s_0}=\{2q_4\}$. By taking $km(F)=2q_4$, we get $s_0:=x^2z+z$ and $F_{s_0}:=\{q_1q_2=q_2q_1\}.$ If $km(F)=q_2q_1$, then $s_0:=z\in F$ and $\texttt{SNF}(s | F)=0.$
\end{enumerate}

These facts show that $s$ reduces strongly over $F$.
\end{example}

The next definition is completely analogous to Definition \ref{Nordbeck1998Definition2} (see also  \cite[Definition 6.6.2]{KreuzerRobbiano2005} and \cite[Definition 4]{Khanetal2019}).

\begin{definition}\label{sagbidefinition}
Let $S$ be a subalgebra of a SPBW extension $A$. A subset $F \subseteq S$ is called a {\rm \texttt{SAGBI}} {\em basis} for $S$ if for all non-zero element $s\in S$, there exists an $F$-monomial $m(F)$ in $\Bbbk\langle F\rangle_{A}$ such that $\operatorname{lm(s)}=\operatorname{lm}(m(F)).$
\end{definition}

As occurs in the setting of free algebras, since monomial orderings are compatible with the multiplication, an equivalent formulation of the definition is that $F$ is a \texttt{SAGBI} basis if the leading term of every non-zero element in $S$ can be written as a product of leading terms of elements in $F$.

\begin{example} Consider $A$ as in the Example \ref{examplered2} and let $\Bbbk\langle F\rangle_{A}$ be the subalgebra generated by the set $$
F=\left\{q_1=zy, q_2=-zy+x, q_3=-y+x\right\}.
$$ 

If we use \texttt{deglex} with $z\succ y\succ x$, then we have that $\operatorname{lm}(F)=\{zy,y\}$. Nevertheless, note that 
$$
f = q_1q_3-q_3q_2=2zyx-z^2+y^2+yx-x^2-z \in S, 
$$

and $\operatorname{lm} (f)=zyx$ cannot be written as a product of terms in $\operatorname{lm}(F)$. This means that $F$ is not a \texttt{SAGBI} basis of $\Bbbk\langle F\rangle_{A}$.
\end{example}

The following result establishes that when $F$ is a \texttt{SAGBI} basis of $S$, then $s \in \Bbbk\langle F\rangle_{A}$ reduces strongly to $s_{0}=0$. This is the corresponding version of Proposition \ref{Nordbeck1998Proposition1} (see also \cite[Theorem 6.6.25 (C1)]{KreuzerRobbiano2005} and \cite[Proposition 2]{Khanetal2019}).

\begin{proposition}\label{prop}
Let $S$ be subalgebra of a SPBW extension $A$ and $F \subseteq S$. If $F$ is a {\rm \texttt{SAGBI}} basis of $S$, then the following assertions hold:
\begin{enumerate}
\item[\rm{(1)}] \label{1} For each  $s\in A$, we have that $s \in S$ if and only {\rm \texttt{SNF}}$(s | F) = 0$.

\item[\rm{(2)}]  $F$ generates the subalgebra $S$, i.e. $S=\Bbbk\langle F\rangle_{A}$.
\end{enumerate}
\end{proposition}
\begin{proof}
\begin{enumerate}
\item[\rm{(1)}]  If $\texttt{SNF}(s | F)=0$, then $s=\sum\limits_{i=1}^{l} k_{i} m_{i}(F)$, where $k_{i} \in \Bbbk$, and hence $s \in S$.

Conversely, suppose that $s \in S$ and $\texttt{SNF}(s | F) \neq 0.$ This means that it cannot be reduced further, i.e. $\operatorname{lm}(\texttt{SNF}(s | F)) \neq {\rm lm}(m(F)),$ for any $F$ -monomial $m(F)$, and this contradicts that $F$ is a \texttt{SAGBI} basis.

\item[\rm{(2)}]  It follows from part (\ref{1}). More exactly, $s \in S$ if and only if $\texttt{SNF}(s | F)=0,$ that is, $s=\sum\limits_{i=1}^{l} k_{i} m_{i}(F)$, with $k_{i} \in \Bbbk.$ Therefore $s \in \Bbbk\langle F\rangle_{A},$ which shows that $S=\Bbbk\langle F\rangle_{A}$.
\end{enumerate}
\end{proof}

\subsection{\texttt{SAGBI} bases criterion}\label{SAGBIbasescriterionSPBW}

We present an algorithm to calculate \texttt{SAGBI} bases in SPBW extensions. 

As in the previous sections, below we present the definition of critical pair (c.f.\cite[Definition 6.6.26]{KreuzerRobbiano2005}, \cite[Definition 2]{Nordbeck2002}, and \cite[Definition 5]{Khanetal2019}).

\begin{definition}
Let $F$ be a subset of a SPBW extension $A$, and $m_1(F),$ $m_2(F)$ two $F$-monomials. The pair given by $\left(m_1(F), m_2(F)\right)$ is called a \textit{critical pair} of $F$ if $\operatorname{lm}(m_1(F))=\operatorname{lm}\left(m_2(F)\right)$. The $T$-{\em polynomial of a critical pair} $\left(m_1(F), m_2(F)\right)$ is defined as 
$$
T\left(m_1(F), m_2(F)\right)= m_1(F)-km_2(F), $$ 

where $k \in \Bbbk$ is such that $\operatorname{lt}(m_1(F)) = k\operatorname{lt}\left(m_2(F)\right).$
\end{definition}

The following definition is precisely our adaptation of Definition \ref{Nordbeck1998Definition1}.

\begin{definition}
Let $F$ be a subset of a SPBW extension $A$, and $\Bbbk\langle F\rangle_{A}$ the  subalgebra generated by $F$ in $A$. Let 
$$
P = \sum\limits_{i=1}^{t} k_{i} m_{i}(F) \in \Bbbk\langle F\rangle_{A}.
$$

The \textit{height} of $P$ is defined as $\operatorname{ht}(P)=\max_{i=1}^{t}\left\{\operatorname{lm}\left(m_{i}(F)\right)\right\},$ where the maximum is taken with respect to one monomial ordering $\preceq$ in $A.$
\end{definition}

The height is defined for a specific representation of elements of $A$, not for the elements itself.

The following result is one of the most important of the chapter. This is the analogous version of  \cite[Theorem 1]{Nordbeck2002} and \cite[Theorem 1]{Khanetal2019} in the case of free algebras and $G$-algebras, respectively.

\begin{theorem}\label{sagbicriterio} 
Suppose that $F$ generates $S$ as a subalgebra in the SPBW extension $A.$ Then $F$ is a {\rm {\texttt{SAGBI}}} basis of $S$ if an only if every $T$-polynomial of every critical pair of $F$ gives zero {\rm {\texttt{SAGBI}}} normal form.
\end{theorem}
\begin{proof}
If $H$ is a \texttt{SAGBI} basis of $S$ then every $T$-polynomial is an element of $S=\Bbbk\langle F\rangle_{A}$ and its \texttt{SAGBI} normal form is equal to zero by part \ref{1} of Proposition \ref{prop}.

Conversely, suppose given $0 \neq s \in S .$ It is sufficient to prove that it has a representation 
$$
s = \sum\limits_{p=1}^{t} k_{p} m_{p}(F),
$$ 

where $k_{p} \in \Bbbk$ and $m_{p}(F) \in \Bbbk\langle F\rangle_{A}$ with 
$$
\operatorname{lm}(s)=\operatorname{ht}\left(\sum\limits_{p=1}^{t} k_{p} m_{p}(F)\right)=\operatorname{lm}(m_i(F)).
$$

Let $s \in S$ given by $s = \sum\limits_{p=1}^{t} k_{p} m_{p}(F)$ with smallest possible height $X$ among all possible representations of $s$ in $S,$ that is, $X=\max _{p=1}^{t}\left\{\operatorname{lm}\left(m_{p}(F)\right)\right\}$. It is clear that $ X \succeq\operatorname{lm}(s)$.

Suppose that $X\succ\operatorname{lm}(s)$ i.e. cancellation of terms occur then there exist at least two $F$-monomials such that their leading monomial is equal to $X .$ Assume we have only two $F$-monomials $m_{i}(F), m_{j}(F)$ in the representation 
$$
s = \sum\limits_{p=1}^{t} k_{p} m_{p}(F)
$$ 

such that 
$$
\operatorname{lm}\left(m_{i}(F)\right)=\operatorname{lm}\left(m_{j}(F)\right)=X.
$$ 

If $T\left(m_{i}(F), m_{j}(F)\right)=m_{i}(F)-k m_{j}(F)$, then 
\begin{align}
s &=\sum_{p=1}^{t} k_{p} m_{p}(F) \\
&=k_{i}\left(m_{i}(F)-k m_{j}(F)\right) +\left(k_{j}+k_{i} k\right) m_{j}(F) \ +\sum_{p=1, p \neq i, j}^{t} k_{p} m_{p}(F)\\
&=k_{i} T\left(m_{i}(F), m_{j}(F)\right)+\left(k_{j}+k_{i} k\right) m_{j}(F) \ +\sum_{p=1, p \neq i, j}^{t} k_{p} m_{p}(F).\label{ec10}
\end{align}

Since $T\left(m_{i}(F), m_{j}(F)\right)$ has a zero \texttt{SAGBI} normal form, this $T$-polynomial is either zero or can be written as sum of $F$-monomials of height $\operatorname{lm}\left(T\left(m_{i}(F), m_{j}(F)\right)\right)$ which is less than $X .$ If $k_{j}+k_{i} k$ is equal to zero, then the right-hand side of expression (\ref{ec10}) is a representation of $s$ with height less than $X,$ which contradicts our initial assumption that we have chosen a representation of $s$ with smallest possible height. Otherwise, the height is preserved, but on the right-hand side of expression (\ref{ec10}), we have only one $F$-monomial $m_{j}(F)$ such that $\operatorname{lm}\left(m_{j}(F)\right)=X,$ which is a contradiction as at least two $F$-monomials of such type must exist in the representation of $s$.
\end{proof}

\begin{remark}\label{NordbeckRemark2Analogue}
From Theorem \ref{sagbicriterio} it follows that every subset $H\subset A$ consisting only of terms or monomials is a \texttt{SAGBI} basis since every $T$-polynomial is clearly equal to zero.
\end{remark}

As we can see the elements to consider for the test of the basis \texttt{SAGBI} are many and sometimes unlimited. For the case of free associative algebras over arbitrary fields, Nordbeck  \cite{Nordbeck1998} gave special conditions on the critical pairs to be taken into account for the construction of a SAGBI basis, reducing the number of such pairs, following we show two examples of SPBW extensions where we show that this construction is not valid in our object of study. From the above, a new question arises: What conditions should we impose so that in the construction of a \texttt{SAGBI} basis the critical pairs to be taken into account are reduced and thus be able to build a finite algorithm?

The following two examples show that \cite[Proposition 6]{Nordbeck1998} for free algebras does not hold for SPBW extensions.

\begin{example}
Consider the Dispin algebra $U(\mathfrak{osp}(1,2))$,  the set $F$ given by $F = \{xy, yz, xz, z^2$, $xy+xz, yz+z^2\}$ and $S=\Bbbk\langle F\rangle_A$. If we use the monomial ordering \texttt{deglex} with $x\succ y\succ z$, then we get the critical pair $(m, m')$ of $F$, where
\begin{align*}
    m = &\ (xy+xz)yz=xy^2z+xyz^2-xz^2,\quad {\rm and}\\
    m' = &\ xy(yz+z^2)=xy^2z+xyz^2.
\end{align*}

Notice that $m=m_1m_2$ and $m'=m'_1m'_2$, with $\operatorname{lm}(m_{i}(F))=\operatorname{lm}(m_{i}^{\prime}(F))$, $i=1,2$. It follows that 
$$
T(m_{1}(F), m_{1}^{\prime}(F)) = xz
$$ 

is a monomial of $S$ with height less than $\operatorname{lm}(m_1(F))=xy$, and $T(m_{2}(F),$ $m_{2}^{\prime}(F))=z^2$ is a monomial of $S$ with height less than $\operatorname{lm}(m_2(F))=yz$. Nevertheless, the $T$-polynomial 
$$
T(m(F), m^{\prime}(F))=-xz^2
$$ 

is not a sum of monomials of $S$, which shows that \cite[Proposition 6]{Nordbeck1998} does not hold in the setting of SPBW extensions.
\end{example}

\begin{example}
Consider the {\em Sklyanin algebra} 
$$
\mathcal{S} = \Bbbk\{x, y, z\} /\langle ayx + bxy + cz^2, axz + bzx + cy^2, azy + byz + cx^2\rangle,
$$

where $a, b, c\in \Bbbk$. If $c=0$ and $a,b\neq 0$ then we obtain the defining relations 
$$
yx = -\displaystyle\frac{b}{a}xy, \quad zx = -\displaystyle\frac{a}{b}xz, \quad {\rm and} \quad zy = -\displaystyle\frac{b}{a}yz.
$$

It can be seen that $\mathcal{S}\cong\sigma(\Bbbk)\langle x,y,z\rangle$ \cite[Theorem 1.14]{ReyesSuarezMomento2017}. Let the set $F\subseteq \mathcal{S}$ given by
\[
F = \{yz+z^2, yz+z, z^2-z\}
\] 

and $S=\Bbbk\langle F\rangle_{\mathcal{S}}$. If we use the monomial ordering \texttt{deglex} with $x\succ y\succ z$, then let $(m, m')$ be a critical pair of $F$, with
\begin{align}
    m = &\ (yz+z^2)(yz+z^2)=-\displaystyle\frac{b}{a}y^2z^2+yz^3+\displaystyle\frac{b^2}{a^2}yz^3+z^4 \quad {\rm and} \\
    m' = &\ (yz+z)(yz+z)=-\displaystyle\frac{b}{a}y^2z^2+yz^2-\displaystyle\frac{b}{a}yz^2+z^2.
\end{align}

Note that $m = m_1m_2$ and $m' = m'_1m'_2$, where $\operatorname{lm}(m_{i}(F))=\operatorname{lm}(m_{i}^{\prime}(F))$, for $i = 1, 2$, and $T(m_{1}(F),$ $m_{1}^{\prime}(F))=z^2-z$ is a monomial of $S$ with height less than $\operatorname{lm}(m_1(F))=yz$; similarly, $T(m_{2}(F),$ $m_{2}^{\prime}(F))=z^2-z$ is a monomial of $S$ with height less than $\operatorname{lm}(m_2(F))=yz$. However, 
$$
T(m(F), m^{\prime}(F))=\left(1+\displaystyle\frac{b^2}{a^2}\right) yz^3+z^4+\left(\displaystyle\frac{b}{a}-1\right)yz^2-z^2
$$ 

is not a sum of monomials of $S$. Once more again, this shows that \cite[Proposition 6]{Nordbeck1998} does not hold.
\end{example}

\section{\texttt{SAGBI} bases under composition}\label{SPBWBasesundercomposition}

In this section, we present the original results about \texttt{SAGBI} bases under composition in the setting of SPBW extensions. We will consider similar definitions to those corresponding investigated by Nordbeck \cite{NordbeckPhD2001} (see also \cite{Nordbeck2002}), and the notation used in Section  \ref{SPBWSAGBIBasicDefinitions}.
\begin{enumerate}
\item [\rm (1)] Let $A = \sigma(R)\left\langle x_{1}, \dotsc, x_{n}\right\rangle$ be a SPBW extension over $R$ and $F$ a finite set of non-zero elements of $A$.

\item [\rm (2)] $\Bbbk\langle F\rangle_{A}$ means {\em the subalgebra of} $A$ {\em generated by} $F$, that is, the polynomial set in the $F$-variables in $A$.

\item [\rm (3)] $m(F)$ denotes a monomial in terms of the elements of $F$, and we call it an $F$-{\em monomial}.

\item [\rm (4)] When we speak of the leading monomial, leading coefficient and leading term of an element in $\Bbbk\langle F\rangle_{A}$, we will always mean the leading monomial, leading coefficient and leading term, respectively, of the element viewed as an element of $A$, relative to the fix monomial ordering in ${\rm Mon}(A)$. Let $\operatorname{lm}(F):=\{\operatorname{lm}(f) \mid f \in F\}$. 

\item [\rm (5)] Let $S$ be a subalgebra of $A$. A subset $H\subset S$ is called a \texttt{SAGBI} basis for $S$ if for every non-zero $s\in S$, there exists a monomial $m$ such that $\widehat{s} = \widehat{m(H)}$.

\item [\rm (6)] Since monomial orderings are preserved by multiplication, \label{(6)}
\[
m(H) = \prod\limits_i h_i\ (h_i \in H)\ {\rm implies}\ \widehat{m(H)} = \prod\limits_i \widehat{h_i}
\]
that is,
\begin{equation}\label{Nordbeck2002(1)Analogue}
\widehat{m(H)} = m(\widehat{H}).    
\end{equation}

This means that an equivalent formulation of \cite[Definition 1]{Nordbeck2002} is that $H$ is a \texttt{SAGBI} basis if the leading term of every non-zero element in $S$ can be written as a product of leading terms of elements in $H$ \cite[Remark 1]{Nordbeck2002}. It can be seen that if $H$ is a \texttt{SAGBI} basis for $S$, then $H$ generates $S$, that is, $S = \Bbbk\langle H\rangle_{A}$. We say that $H$ is a \texttt{SAGBI} basis meaning that $H$ is a \texttt{SAGBI} basis for $\Bbbk\langle H\rangle_{A}$.

\item [\rm (7)] We say that two monomials $m, m'$ form a {\em critical pair} $(m, m')$ of $H$ if $\widehat{m(H)} = \widehat{m'(H)}$. If $c\in \Bbbk$ is such that $m(H)$ and $cm'(H)$ have the same leading coefficient, then we define the $T$-{\em polynomial} of $(m, m')$ \cite[Definition 2]{Nordbeck2002} as 
$$
T(m, m') = m(H) - cm'(H).
$$

The idea with the constant $c$ is that the leading terms cancel in $T(m, m')$, whence $\widehat{T(m, m')} < \widehat{m(H)} = \widehat{m'(H)}$.

\item [\rm (8)] Every subset $H\subset A$ consisting only of terms (or terms times coefficients) is a \texttt{SAGBI} basis.

\item [\rm (9)] Let $\Theta = \{\theta_1, \dotsc, \theta_n\}$ be a subset of $A$, and let $f\in A$. The {\em composition} of $f$ by $\Theta$, written $f\circ \Theta$, is the polynomial obtained from $f$ by replacing each ocurrence of the $x_i$ with $\theta_i$. For a subset $F\subset A$, $F\circ \Theta := \left\{f\circ \Theta \mid f\in F\right\}$. Once more again, it is assumed that $\theta_i \notin \Bbbk$ for all $i$. This is used to guarantee that $X\neq 1$ implies $X\circ \widehat{\Theta} \neq 1$, for every $X\in {\rm Mon}(A)$. It is straightforward to see that two forms of compositions are associative in the following sense:
\begin{equation}\label{Nordbeck2002(3)analogue}
m(H) \circ \Theta = m(H\circ \Theta).
\end{equation}

\item [\rm (10)] Notice that the notion $m(H\circ \Theta)$ makes sense due to the natural correspondence between the sets $H = \left\{h_1, h_2, \dotsc, \right\}$ and $H\circ \Theta = \left\{h_1 \circ \Theta, h_2 \circ \Theta, \dotsc\right\}$. Note also that for elements $f, g \in A$, 
\begin{align}
    (fg) \circ \Theta = &\ f\circ \Theta g \circ \Theta, \label{Nordbeck2002(4)analogue} \\
    (f + g)\circ \Theta = &\ f\circ \Theta + g \circ \Theta \label{Nordbeck2002(5)analogue} 
\end{align}
    
Having in mind that the order is preserved by multiplication, in a similar way to \cite[Remark 1]{Nordbeck2002}, we get that 
\begin{equation}\label{Nordbeck2002(6)analogue}
    \widehat{X \circ \Theta} = X \circ \widehat{\Theta},\quad {\rm for\ every}\ X\in {\rm Mon}(A).
\end{equation}

\item [\rm (11)] Let $\Theta = \left\{\theta_1, \dotsc, \theta_n\right\}$ be a subset of $A$. Composition by $\Theta$ is {\em compatible with the monomial ordering} $\prec$ if for all monomials $X_i, X_j\in {\rm Mon}(A)$, we have that 
\begin{equation}\label{Nordbeck2002(7)analogue}
X_i \prec X_j \quad {\rm implies}\quad X_i \circ \widehat{\Theta} \prec X_j \circ \widehat{\Theta}.
\end{equation}

For an element $f\in A$ written as a linear combination of monomials in decreasing order $f = \sum\limits_{i = 1}^t r_iX_i$, $X_1 \succ \dotsb \succ X_t$ (Remark \ref{notesondefsigampbw}  (iv)), if composition by $\Theta$ is compatible with the monomial ordering $\prec$, then $X_1 \circ \widehat{\Theta} \succ \dotsb \succ X_s \circ \widehat{\Theta}$, so expressions (\cite[Remark 4 (5) and (6)]{Nordbeck2002} guarantee that
\begin{equation}\label{Nordbeck2002(8)analogue}
    \widehat{f \circ \Theta} = \widehat{f} \circ \widehat{\Theta}.
\end{equation}

Composition by $\Theta$ is {\em compatible with nonequality} if for all monomials $X_i, X_j \in {\rm Mon}(A)$, we have
\begin{equation}
    X_i\neq X_j\quad {\rm implies}\quad X_i \circ \widehat{\Theta} \neq X_j\circ \widehat{\Theta}.
\end{equation}

Since monomial orderings are total, if composition by $\Theta$ is compatible with the monomial ordering $\prec$, then composition by $\Theta$ is compatible with nonequality \cite[Lemma 1]{Nordbeck2002}.
\end{enumerate}

Lemma \ref{Nordbeck2002Lemma2Analogue} and Proposition \ref{Nordbeck2002Proposition1Analogue} are the analogues of \cite[Lemma 2 and Proposition 1]{Nordbeck2002}, respectively. We need these two results to prove the sufficiency of the compatibility with the ordering in Theorem \ref{Nordbeck2002Theorem4Analogue}.

\begin{lemma}\label{Nordbeck2002Lemma2Analogue}
Suppose that the composition by $\Theta$ is compatible with the ordering $\prec$. If $(m, m')$ is a critical pair of $H \circ \Theta$, then $(m, m')$ is also a critical pair of $H$.
\end{lemma}
\begin{proof}
Suppose that $(m, m')$ is a critical pair of $H \circ \Theta$. By (\ref{Nordbeck2002(3)analogue}), 
$$
m(H \circ \Theta) = m(H) \circ \Theta \quad {\rm and} \quad m'(H\circ \Theta) = m'(H) \circ \Theta, 
$$ 

whence ${\rm lt}(m (H \circ \Theta)) = {\rm lt} (m'(H\circ \Theta))$ and so ${\rm lt}(m(H)) \circ {\rm lt}(H) = {\rm lt}(m'(H)) \circ {\rm lt}(\Theta)$ due to (\ref{Nordbeck2002(6)analogue}). Now, by assumption the composition is compatible with the ordering, and hence with the nonequality, which implies that ${\rm lt}(m(H)) = {\rm lt} (m'(H))$, i.e. $(m, m')$ is a critical pair of $H$.
\end{proof}

The following result illustrates the necessity of the compatibility with the ordering.

\begin{proposition}\label{Nordbeck2002Proposition1Analogue}  
If composition by $\Theta$ is compatible with the ordering $\prec$, then composition by $\Theta$ commutes with noncommutative {\rm \texttt{SAGBI}} bases computation.
\end{proposition}
\begin{proof}
Consider an arbitrary \texttt{SAGBI} basis $H$. We have to show that $H\circ \Theta$ is also a \texttt{SAGBI} basis. Consider an arbitrary critical pair of $H\circ \Theta$. By Lemma \ref{Nordbeck2002Lemma2Analogue} we know that $(m, m')$ is also a critical pair of $H$. Theorem \ref{sagbicriterio} guarantees the expression
\begin{equation}\label{Nordbeck2002(10)Analogue}
    m(H) - cm'(H) = \sum_i c_i m_i(H)\quad {\rm (or\ zero)},\quad \widehat{m_i(H)} \prec \widehat{m(H)} = \widehat{m'(H)},\quad {\rm for\ all}\ i.
\end{equation}

If we compose the $T$-polynomial by $\Theta$, then expressions (\ref{Nordbeck2002(3)analogue}), (\ref{Nordbeck2002(4)analogue}), and (\ref{Nordbeck2002(5)analogue}) guarantee
\begin{equation}\label{Nordbeck2002(11)Analogue}
    m(H\circ \Theta) - cm'(H \circ \Theta) = \sum_i c_im_i(H\circ \Theta)\quad {\rm (or\ zero)},
\end{equation}

Now, if we compose the inequality in expression (\ref{Nordbeck2002(10)Analogue}) by $\widehat{\Theta}$, by (\ref{Nordbeck2002(6)analogue}) and (\ref{Nordbeck2002(7)analogue}) we get
\begin{equation}\label{Nordbeck2002(12)Analogue}
    \widehat{m_i(H\circ \Theta)} \prec \widehat{m(H\circ \Theta)} = \widehat{m'(H\circ \Theta)}, \quad {\rm for\ all}\ i.
\end{equation}

Notice that the leading terms in the left-hand side of expression (\ref{Nordbeck2002(11)Analogue}) cancel, so the constant $c$ must be the same as in the definition of the $T$-polynomial of $(m, m')$ with respect to $H\circ \Theta$. Therefore, expressions (\ref{Nordbeck2002(11)Analogue}) and (\ref{Nordbeck2002(12)Analogue}) are a representation in the sense of Theorem \ref{sagbicriterio}, and having in mind that the critical pair $(m, m')$ of $H\circ \Theta$ was arbitrary, it follows that $H\circ \Theta$ is a \texttt{SAGBI} basis.
\end{proof}

Now, the proof of the necessity of the compatibility with the ordering in Theorem \ref{Nordbeck2002Theorem4Analogue} requires of Lemma \ref{Nordbeck2002Lemma3Analogue} and Proposition \ref{Nordbeck2002Proposition2Analogue} which are the analogue versions of \cite[Lemma 3 and Proposition 2]{Nordbeck2002}, respectively.

\begin{lemma}\label{Nordbeck2002Lemma3Analogue}
If $u, v$ are two monomials  with $u\neq v$ but $u\circ \widehat{\Theta} = v\circ \widehat{\Theta}$, then for every $w\prec u$, $H = \{u - w, v\}$ is a {\rm \texttt{SAGBI}} basis.
\end{lemma}
\begin{proof}
It is clear that both $u$ and $v$ must be different from $1$. If this is not the case, say $v = 1$, then $u\neq 1$, and having in mind that the elements $\theta_i$'s are nonconstant, $v\circ \widehat{\Theta} = u\circ \widehat{\Theta} \neq 1$, a contradiction.

The idea is to show that $H$ has no non-trivial critical pairs, that is, if $\widehat{m(H)} = \widehat{m'(H)}$ then $m = m'$, because in this situation $H$ is a \texttt{SAGBI} basis, and hence every $T$-polynomial must be identically equal to zero. Let us see the proof by contradiction.

Let $(m, m')$ be a non-trivial arbitrary critical pair of $H$. Notice that $\widehat{H} = \{u, v\}$, and so $\widehat{m(H)} = m(\widehat{H}) = u^k v^l$ and $\widehat{m'(H)} = u^s v^t$, and $u^k v^l = u^s v^t$. Since the critical pair is non-trivial, then $k\neq s$ and $l\neq t$, which means that $k > s$ and $l < t$ or vice versa. In this way, it follows that $u^{i} = v^{j}$, with $i, j \ge 0$. If we compose this equality by $\widehat{\Theta}$, expression (\ref{Nordbeck2002(4)analogue}) implies that $(u\circ \widehat{\Theta})^{i} = (v\circ \widehat{\Theta})^{j}$. By using that $u\circ \widehat{\Theta} = v\circ \widehat{\Theta} \neq 1$, it follows that $i = j$. Finally, the equality $u^{i} = v^{j}$ implies $u = v$, a contradiction, whence the critical pair $(m, m')$ is trivial. We conclude that $H$ is a \texttt{SAGBI} basis.
\end{proof}

\begin{proposition}\label{Nordbeck2002Proposition2Analogue}
If composition by $\Theta$ commutes with noncommutative {\rm \texttt{SAGBI}} bases computation, then composition by $\Theta$ is compatible with nonequality.
\end{proposition}
\begin{proof}
Let $\Theta$ be commuting with \texttt{SAGBI} bases computation. Once more again, we proceed by contradiction. 

Suppose that there exist two different monomials $u, v \in {\rm Mon}(A)$ but $u\circ \widehat{\Theta} = v \circ \widehat{\Theta}$. As we saw in the proof of Lemma \ref{Nordbeck2002Lemma3Analogue}, $u, v \neq 1$. Since every subset $H\subset A$ consisting only of monomials (or terms) is a \texttt{SAGBI} bases, $H = \{u, v\}$ is a \texttt{SAGBI} basis, and so $H \circ \Theta = \{u\circ \Theta, v\circ \Theta\}$ also is. Hence, if 
$$
f = u\circ \Theta - v\circ \Theta \in \Bbbk\langle H \circ \Theta \rangle_{A}
$$ 

is not equal to 1 or zero, then the assertion follows. Since $\widehat{f} \prec u \circ \widehat{\Theta} = v \circ \widehat{\Theta}$, $\widehat{f}$ cannot be written as a product from $\widehat{H \circ \Theta} = \{u\circ \widehat{\Theta}, v\circ \widehat{\Theta}\}$, which means that $H\circ \Theta$ cannot be a \texttt{SAGBI} basis \cite[Remark 1]{Nordbeck2002}.

Let $u' = ux_i$ and $v' = vx_i$ for some indeterminate $x_i$ in the SPBW extension of $A$. Note that $u' \circ \Theta = (u \circ \Theta) \theta_i$ and 
\begin{equation}\label{Nordstrom2002(13)Analogue}
    u' \neq v'\quad {\rm and}\quad \widehat{u'\circ \Theta} = \widehat{v' \circ \Theta}.
\end{equation}

If $f = u\circ \Theta - v\circ \Theta = 1$, then we consider $H' = \{u', v'\}$ which is a \texttt{SAGBI} basis by Remark \ref{NordbeckRemark2Analogue}. Now,
\[
f' = u'\circ \Theta - v'\circ \Theta = (u \circ \Theta)\theta_i - (v\circ \Theta)\theta_i = f\theta_i = \theta_i \in \Bbbk\langle H' \circ \Theta \rangle_{A},
\]

and, once more again, it follows that $H = \{u' \circ \Theta, v'\circ \Theta\}$ cannot be a \texttt{SAGBI} basis (note that $\widehat{\theta_i} \prec \widehat{u'\circ \Theta} = \widehat{v' \circ \Theta}$).

We only need to consider the case $f = u\circ \Theta - v\circ \Theta = 0$ (e.g. if $\Theta = \widehat{\Theta}$). Let $H' = \{u' + x_i, v'\}$. By expression (\ref{Nordstrom2002(13)Analogue}) and $x_i \prec u' = ux_i$, it follows that $H'$ is a \texttt{SAGBI} basis by Lemma \ref{Nordbeck2002Lemma3Analogue}. As above, we obtain a contradiction from $f' = u' \circ \Theta - v' \circ \Theta = \theta_i \in \Bbbk\langle H' \circ \Theta \rangle_{A}$.
\end{proof}

\begin{proposition}\label{Nordbeck2002Proposition3Analogue}
If composition by $\Theta$ commutes with noncommutative {\rm \texttt{SAGBI}} bases computation, then composition by $\Theta$ is compatible with the ordering $\prec$.   
\end{proposition}
\begin{proof}
Suppose that composition by $\Theta$ commutes with \texttt{SAGBI} bases computation. Let $u, v \in {\rm Mon}(A)$ two monomials with $u\succ v$. We want to show that $u \circ \widehat{\Theta} \succ v\circ \widehat{\Theta}$. Due to $u \neq v$, Proposition \ref{Nordbeck2002Proposition2Analogue} shows that we cannot have $u \circ \widehat{\Theta} = v \circ \widehat{\Theta}$, which means that we exclude the case $u \circ \widehat{\Theta} \prec v\circ \widehat{\Theta}$.

By Remark \ref{NordbeckRemark2Analogue} we know that $H' = \{u, v\}$ is a \texttt{SAGBI} basis, and due to (\ref{Nordbeck2002(1)Analogue}), $\Bbbk \langle H\rangle_A = \Bbbk \langle H' \rangle_A$, and $\widehat{H} = \widehat{H}'$, we get that $H = \{u - v, v\}$ is a \texttt{SAGBI} basis. Thus, $H \circ \Theta = \{u \circ \Theta - v \circ \Theta, v\circ \Theta\}$ must be also a \texttt{SAGBI} basis.

Consider $u\circ \widehat{\Theta} \prec v \circ \widehat{\Theta}$. It follows that 
$$\widehat{H \circ \Theta} = \{v \circ \widehat{\Theta}\} \quad {\rm and} \quad u\circ \Theta = (u\circ \Theta - v\circ \Theta) + v\circ \Theta \in \Bbbk \langle H \circ \Theta \rangle_A.
$$ 

However, as we saw in the proof of Proposition \ref{Nordbeck2002Proposition2Analogue}, $u \circ \widehat{\Theta} \prec v \circ \widehat{\Theta}$, and so $u \circ \widehat{\Theta} \neq 1$ cannot be expressed as a power of $v \circ \widehat{\Theta}$ which means that $H \circ \Theta$ is not a \texttt{SAGBI} basis. In other words, the assumption $u \circ \widehat{\Theta} \prec v \circ \widehat{\Theta}$ is false, and hence the composition by $\Theta$ is compatible with the ordering $\prec$.
\end{proof}

Finally, we present the most important result of this section that follows from Propositions \ref{Nordbeck2002Proposition1Analogue} and \ref{Nordbeck2002Proposition3Analogue}.

\begin{theorem}\label{Nordbeck2002Theorem4Analogue}
    Composition by $\Theta$ commutes with noncommutative {\rm \texttt{SAGBI}} bases computation if and only if the composition is compatible with the ordering $\prec$.
\end{theorem}
\begin{proof}
	The sufficiency is guaranteed by Lemma \ref{Nordbeck2002Lemma2Analogue}
 and Proposition \ref{Nordbeck2002Proposition1Analogue}. 
 
 With respect to the proof of the necessity, this follows from Propositions  \ref{Nordbeck2002Proposition2Analogue} and \ref{Nordbeck2002Proposition3Analogue} and the remark presented by Nordbeck \cite[Theorem 4]{Nordbeck2002} applied to Lemma \ref{Nordbeck2002Lemma3Analogue}.
\end{proof}

\section{Conclusions and future work}\label{conclusionsfuturework}

 Regarding the computational development of \texttt{SAGBI} bases, this has been strongly driven in the commutative case as we can see in the literature (e.g.  \cite{Hashemietal2013, KreuzerRobbiano2005, Pfister} and references therein), and recently in Robbiano and Bigatti's paper \cite{RobbianoBigatti2022}. Since Fajardo in his Ph.D. Thesis \cite{Fajardo2018PhD} and related papers \cite{Fajardo2019, Fajardo2022}, \cite[Appendices C, D, E]{Fajardoetal2020} developed the library \texttt{SPBWE.lib} in \texttt{Maple} with the aim of making homological and constructive computations of Gr\"obner bases of SPBW extensions, it is natural to consider the problem of developing a computational approach to the \texttt{SAGBI} basis theory for SPBW extensions following the ideas presented in Section \ref{SAGBIbasesSPBW}.
 
As we said in Section \ref{SAGBIbasesSPBW}, Gallego developed the theory of Gr\"obner bases of SPBW extensions. However, the problem of Gr\"obner bases under composition of these extensions has not been investigated, so a first natural task is to investigate this problem. In fact, recently Kanwal and Khan \cite{KanwalKhan2023} considered the question of \texttt{SAGBI}-Gr\"obner bases under composition. This paper is of our interest in the near future.

Last but not least, related with the topic of Gr\"obner bases, another interesting questions concern the notions of {\em Universal Gr\"obner basis} and the {\em Gr\"obner Walk method}. In the commutative setting, a set $F$ which is a Gr\"obner basis for an ideal $I$ of the commutative polynomial ring $\Bbbk[x_1, \dotsc, x_n]$ with respect to every monomial ordering is called a {\em universal Gr\"obner basis}. From \cite[Exercise 1.8.6(d)]{AdamsLoustaunau1994} we know that every ideal of $\Bbbk[x_1, \dotsc, x_n]$ has a basis of this kind. Since some approaches to this notion of basis in the noncommmutative case have been developed by Weispfenning \cite{Weispfenning1989}, we consider interesting to ask the possibility of these bases in the setting of SPBW extensions. With respect to the Gr\"obner Walk method, this is a basis conversion method proposed by Collart, Kalkbrener, and Mall \cite{Collartetal1997} with the aim of converting a given Gr\"obner basis, respect to a fixed monomial order, of a polynomial ideal $I$ of the commutative polynomial ring $\Bbbk[x_1, \dotsc, x_n]$ to a Gr\"obner basis of $I$ with respect to another monomial order. Since some approaches have been realized to Gr\"obner Walk method in the noncommutative setting (for instance, Evans in his Ph.D.  Thesis \cite{Evans2005} who presented an algorithm for a noncommutative Gr\"obner Walk in the case of conversion between two harmonious monomial orderings), it is important to us to investigate the possibility of this procedure in the setting of Gr\"obner bases of SPBW extensions.

\section{Acknowledgments}

We would like to express our sincere thanks to Professor Oswaldo Lezama for introducing us to this project.


\begin{thebibliography}{}

\bibitem{Apel1988} J. Apel. Gr\"obnerbasen in Nichtkommutativen Algebren und ihre Anwendung. Ph.D. thesis, Universit\"at Leipzig (1988).

\bibitem{AdamsLoustaunau1994} W. W. Adams, P. Loustaunau. An {I}ntroduction to {G}r\"obner {B}ases. Graduate Studies in Mathematics. American Mathematical Society (1994).

\bibitem{Bavula2023} V. V. Bavula. Description of bi-quadratic algebras on 3 generators with PBW basis. {\em J. Algebra} 631 (2023) 695--730.

\bibitem{BellSmith1990} A. D. Bell, S. P. Smith. Some 3-Dimensional Skew Polynomial Rings. University of Wisconsin, Milwaukee (1990). (preprint)

\bibitem{BellGoodearl1988} A. Bell, K. Goodearl. Uniform rank over differential operator rings and {P}oincar\'e-{B}irkhoff-{W}itt extensions. {\em Pacific J. Math.} 131 (1) (1988) 13--37.

\bibitem{BeckerWeispfenning1993} T. Becker, V. Weispfenning. Gr\"obner Bases. A Computational Approach to Commutative Algebra. Graduate Texts in Mathematics, Vol. 141, Springer--Verlag (1993).


\bibitem{Buchberger1965} B. Buchberger. Ein Algorithms zum Auffinden der Basiselemente des Restklassenringes nach einem nulldimensionalen Polynomideal. Ph.D. thesis, University of Innsbruck, Innsbruck, Austria (1965).

\bibitem{BuesoTorrecillasVerschoren} J. Bueso, J. G\'omez-Torrecillas, A. Verschoren. Algorithmic {M}ethods in {N}on-commutative {A}lgebra: {A}pplications to {Q}uantum {G}roups. Dordrecht, Kluwer (2003).

\bibitem{Collartetal1997} M. Collart, M. Kalkbrener, D. Mall. Converting {B}ases with the {G}r\"obner {W}alk. {\em J. Symbolic Comput.} 24(3--4) (1997) 465--469.

\bibitem{CoxLittleOshea2015} D. A. Cox, J. Little, D. O'Shea. Ideals, {V}arieties, and {A}lgorithms. {A}n {I}ntroduction to {C}omputational {A}lgebraic {G}eometry and {C}ommutative {A}lgebra. {F}ourth {E}dition, Undergraduate Texts in Mathematics, Springer Cham (2015).

\bibitem{Evans2005} G. A. Evans. Noncommutative Involutive Bases. Ph.D. thesis, University of Wales, Bangor (2005).

\bibitem{Fajardo2018PhD} W. Fajardo. Extended modules over skew {P}{B}{W} extensions. Ph.D. thesis, Universidad Nacional de Colombia, Bogot\'a, D. C.,  Colombia (2018).

\bibitem{Fajardo2019} W. Fajardo. A {C}omputational {M}aple {L}ibrary for {S}kew {P}{B}{W}  {E}xtensions. {\em Fund. Inform.} 167 (3) (2019) 159--191.

\bibitem{Fajardo2022} W. Fajardo. Right {B}uchberger algorithm over bijective skew {P}{B}{W} extensions. {\em Fund. Inform.} 184(2) (2022) 83--105.

\bibitem{Fajardoetal2020} W. Fajardo, C. Gallego, O. Lezama, A. Reyes, H. Su\'arez, H. Venegas. Skew PBW Extensions. Ring and Module-theoretic Properties, Matrix and Gr\"obner Methods, and Applications. Springer, Cham (2020).

\bibitem{Gallego2015PhD} C. Gallego. Matrix methods for projective modules over $\sigma$-{P}{B}{W} extensions. Ph.D. thesis, Universidad Nacional de Colombia, Bogot\'a, D. C., Colombia (2015).

\bibitem{Gallego2016FG} C. Gallego. Filtered-graded transfer of noncommutative {G}r\"obner bases. {\em Rev. Colombiana Mat.} 50 (1)  (2016) 41--54.

\bibitem{Gallego2016}C. Gallego. Matrix computations on projective modules using noncommutative {G}r\"obner bases. {\em Journal of Algebra, Number Theory: Advances and Applications} 15(2) (2016) 101--139.

\bibitem{GallegoLezama2010} C. Gallego, O. Lezama. Gr\"obner {B}ases for {I}deals of {$\sigma$}-{P}{B}{W} {E}xtensions. {\em Comm. Algebra} 39 (1) (2011) 50--75.

\bibitem{GallegoLezama2017} C. Gallego, O. Lezama. Projective modules and {G}r\"obner bases for skew {P}{B}{W} extensions. {\em Dissertationes Math.} 521 (2017) 1--50.

\bibitem{GoodearlWarfield2004} K. R. Goodearl, R. B. Warfield Jr. An {I}ntroduction to {N}oncommutative {N}oetherian {R}ings. Cambridge University Press. London (2004).

\bibitem{GomezTorrecillas2014} J. G\'omez-Torrecillas. Basic {M}odule {T}heory over {N}on-commutative {R}ings with {C}omputational {A}spects of {O}perator {A}lgebras. In M. Barkatou, T. Cluzeau, G. Regensburger, and M. Rosenkranz, editors, Algebraic and Algorithmic Aspects of Differential and Integral Operators. AADIOS 2012, volume 8372 of Lecture Notes in Computer Science, pages 23--82. Berlin, Heidelberg: Springer (2014).

\bibitem{Gordan1900} P. M. Gordan. Les Invariants Des Formes Binaires. {\em J. Math. Pures Appl.} (9) 5(6) (1900) 141--156.

\bibitem{Pfister} G. M. Greuel, G. Pfister. A Singular Introduction to Commutative Algebra. Second edition. Springer--Verlag Berlin Heidelberg (2008).

\bibitem{Hashemietal2013} A. Hashemi, B. M.-Alizadeh, M. Riahi. Invariant G$^2$V algorithm for computing {S}{A}{G}{B}{I}-{G}r\"obner bases. {\em Sci. China Math.} 56 (9) (2013) 1781--1794.

\bibitem{HavlicekKlimykPosta2000} M. Havl\'i\v{c}ek, A. U. Klimyk, and S. Po\v{s}ta. Central elements of the algebras ${U}'(\mathfrak{so}_m)$ and ${U}(\mathfrak{iso}_m)$. {\em Czech. J. Phys.} 50 (1) (2000) 79--84.

\bibitem{Jordan1990} D. A. Jordan. Down-Up algebras and Ambiskew polynomial rings. {\em J. Algebra} 228 (1) (2000) 311--346.

\bibitem{Jordan2001} D. A. Jordan. The Graded Algebra Generated by Two Eulerian Derivatives. {\em Algebr. Represent. Theory} 4 (3) (2001) 249--275.

\bibitem{KandriRodyWeispfenning1990} A. Kandri-Rody, V. Weispfenning. Non-commutative Gr\"obner bases in algebras of solvable type. {\em J. Symbol. Comput.}  9 (1) (1990) 1--26.

\bibitem{KapurMadlener1989} D. Kapur and K. Madlener. A Completion Procedure for Computing a Canonical Basis for a $\Bbbk$-subalgebra. In Proceedings of the third conference on Computers and Mathematics, pages 1--11. Springer (1989). 

\bibitem{Khanetal2019} M. A. B. Khan, J. A. Khan, M. A. Binyamin. SAGBI Bases in $G$-Algebras. {\em Symmetry} 11 (2) (2019) 221--231.

\bibitem{KanwalKhan2023} N. Kanwal, J. A. Khan. Sagbi-{G}r\"obner {B}ases {U}nder {C}omposition. {\em J. Syst. Sci. Complex.} 36 (2023) 2214--2224.

\bibitem{KreuzerRobbiano2005} M. Kreuzer, L. Robbiano. Computational Commutative Algebra, Vol. 2. Springer (2005).

\bibitem{LezamaMarin2009} O. Lezama and V. Mar\'in. Una aplicaci\'on de las bases SAGBI - Igualdad de Sub\'algebras (caso cuerpo). {\em Revista Tumbaga} 1 (4) (2009) 31--41.

\bibitem{LezamaReyes2014} O. Lezama and A. Reyes. Some {H}omological {P}roperties of {S}kew {P}{B}{W} {E}xtensions. {\em Comm. Algebra} 42 (3) (2014) 1200--1230.

\bibitem{Levandovskyy2005} V. Levandovskyy. Non-Commutative Computer Algebra for Polynomial Algebras: Gr\"obner Bases, Applications and Implementation. Ph.D. thesis, Universit\"at Kaiserslautern (2005).

\bibitem{Li2002} H. Li. Noncommutative {G}r\"obner {B}ases and {F}iltered-{G}raded {T}ransfer. Lect. Notes in Math., Springer Berlin Heidelberg (2002).

\bibitem{Miller1998} L. Miller. Effective Algorithms for Intrinsically Computing SAGBI-Gr\"obner Bases in a Polynomial Ring over a Field. In B. Buchberger and F. Winkler (eds.), Gr\"obner Bases and Applications, London Math. Soc. Lecture Note Ser. 251, pages 421--433 (1998).

\bibitem{Nordbeck1998} P. Nordbeck. Canonical {S}ubalgebraic {B}ases in {N}on-commutative {P}olynomial {R}ings. In  Proceedings of the 1998 International Symposium on Symbolic and Algebraic Computation, pages 140--146  (1998).

\bibitem{NordbeckPhD2001} P. Nordbeck. Canonical {B}ases for {A}lgebraic {C}omputations. Ph.D. thesis, Lund Institute of Technology, Lund University, Sweeden, (2001).

\bibitem{Nordbeck2002}P. Nordbeck. S{A}{G}{B}{I} {B}ases {U}nder {C}omposition. {\em J. Symbolic Comput.} 33 (1) (2002) 67--76.

\bibitem{Ore1933} O. Ore. Theory of non-commutative polynomials. {\em Ann. Math. 2} 34 (3) (1933) 480--508. 

\bibitem{Reyes2014} A. Reyes. Jacobson's conjecture and skew PBW extensions. {\em Rev. Integr. Temas Mat.} 32 (2) (2014) 139--152.

\bibitem{ReyesSuarezMomento2017} A. Reyes, H. Su\'arez. Bases for {Q}uantum {A}lgebras and {S}kew {P}oincar\'e-{B}irkhoff-{W}itt {E}xtensions. {\em Momento} 54 (1) (2017) 54--75.

\bibitem{ReyesSuarez20173dimentional} A. Reyes, H. Su{\'{a}}rez. PBW Bases for Some 3-Dimensional Skew Polynomial Algebras. {\em Far East J. Math. Sci.} 101 (6) (2017) 1207--1228.

\bibitem{RobbianoBigatti2022} L. Robbiano, A. M. Bigatti. Saturations of subalgebras, {S}{A}{G}{B}{I} bases, and {U}-invariants. {\em J. Symbolic Comput.} 109 (2022) 259--282.

\bibitem{RobbianoSweedler1990} L. Robbiano, M. Sweedler. Subalgebra bases. In W. Bruns and A. Simis (eds.), Commutative algebra, Proc. Workshop Salvador 1988, Lect. Notes in Math. 1430, pages 61--87. Springer, Berlin (1990).

\bibitem{ShannonSweedler1989} D. Shannon, M. Sweedler. Using Gr\"obner bases to determine algebra membership, split surjective algebra homomorphisms determine birational equivalence. {\em J. Symbolic Comput.} 6 (2-3) (1989) 267--273.

\bibitem{Seiler2010} W. M. Seiler. Involution. The Formal Theory of Differential Equations and its Applications in Computer Algebra. Algorithms and Computation in Mathematics, Vol. 24. Springer, Berlin (2010).

\bibitem{SuarezPhDThesis2023} Y. Su\'arez. Involutive and SAGBI bases for skew PBW extensions. Ph.D. thesis, Universidad Nacional de Colombia, Bogot\'a, D. C., Colombia (2023).

\bibitem{Weispfenning1989} V. Weispfenning. Constructing universal {G}r\"obner bases. In L. Huguet and A. Poli (eds), Applied Algebra, Algebraic Algorithms and Error-Correcting Codes. AAECC 1987, volume 356 of Lecture Notes in Computer Science, pages 13--70. Springer, Berlin, Heidelberg (1989).

\bibitem{Zhedanov1991} A. S. Zhedanov. \textquotedblleft Hidden symmetry\textquotedblright\ of Askey-Wilson polynomials. {\em Theoret. and Math. Phys.} 89 (2) (1991) 1146--1157.

\end{thebibliography}
\end{document}